\documentclass[a4paper,11pt]{article}

\usepackage[utf8]{inputenc} 
\usepackage[english]{babel}
\usepackage[left=1.5in,textwidth=30pc,textheight=54.5pc]{geometry} 

\usepackage{pgfplots}
\pgfplotsset{compat=1.16}
\usepackage[T1]{fontenc} 

\usepackage{dsfont}

\usepackage{amsmath,amsfonts,amssymb,amsthm,mathrsfs,dsfont}
\usepackage{mathtools}
\usepackage[hidelinks]{hyperref}

\usepackage[capitalize,nameinlink]{cleveref}

\usepackage{paralist}
\usepackage[shortlabels]{enumitem}
\usepackage{textcase} 
\usepackage{filemod} 
\usepackage[longnamesfirst]{natbib}

\setcitestyle{authoryear,open={\lparen},close={\rparen}}
\setcitestyle{citesep={,}}

\usepackage{titlesec} 
\usepackage{etoolbox} 

\usepackage{bm}
\makeatletter
\DeclareBoldMathCommand{\bbd}{D}
\def\cite{\citet}

\numberwithin{equation}{section}
\allowdisplaybreaks

\def\@noindentfalse{\global\let\if@noindent\iffalse}
\def\@noindenttrue {\global\let\if@noindent\iftrue}
\def\@aftertheorem{%
  \@noindenttrue
  \everypar{%
    \if@noindent%
      \@noindentfalse\clubpenalty\@M\setbox\z@\lastbox%
    \else%
      \clubpenalty \@clubpenalty\everypar{}%
    \fi}}

\theoremstyle{plain}
\newtheorem{theorem}{Theorem}[section]
\AfterEndEnvironment{theorem}{\@aftertheorem}
\newtheorem{lemma}[theorem]{Lemma}
\AfterEndEnvironment{lemma}{\@aftertheorem}
\newtheorem{corollary}[theorem]{Corollary}
\AfterEndEnvironment{corollary}{\@aftertheorem}

\AfterEndEnvironment{proposition}{\@aftertheorem}

\theoremstyle{definition}
\newtheorem{definition}[theorem]{Definition}
\AfterEndEnvironment{definition}{\@aftertheorem}
\newtheorem{remark}[theorem]{Remark}
\AfterEndEnvironment{remark}{\@aftertheorem}

\AfterEndEnvironment{example}{\@aftertheorem}    
\AfterEndEnvironment{proof}{\@aftertheorem}

\titleformat{name=\section}
  {\center\small\sc}{\thesection\kern1em}{0pt}{\MakeTextUppercase}
\titleformat{name=\section, numberless}
  {\center\small\sc}{}{0pt}{\MakeTextUppercase}

\titleformat{name=\subsection}
  {\bf\mathversion{bold}}{\thesubsection\kern1em}{0pt}{}
\titleformat{name=\subsection, numberless}
  {\bf\mathversion{bold}}{}{0pt}{}

\titleformat{name=\subsubsection}
  {\normalsize\itshape}{\thesubsubsection\kern1em}{0pt}{}
\titleformat{name=\subsubsection, numberless}
  {\normalsize\itshape}{}{0pt}{}

\usepackage[linewidth=1pt]{mdframed}
\usepackage[textwidth=1.2in]{todonotes}
\let\@@todo\todo
\def\todo#1{\@@todo[color=red,backgroundcolor=red!10,size=\tiny]{#1}}
\def\checked#1{\@@todo[color=green!50!black,backgroundcolor=green!10,size=\tiny]{#1}}

\def\be#1{\begin{equation*}#1\end{equation*}}
\def\ben#1{\begin{equation}#1\end{equation}}
\def\bes#1{\begin{equation*}\begin{aligned}#1\end{aligned}\end{equation*}}
\def\besn#1{\begin{equation}\begin{aligned}#1\end{aligned}\end{equation}}
\def\bs#1{\begin{equation}\begin{split}#1\end{split}\end{equation}}

\def\bmn#1{\begin{multline}#1\end{multline}}

\usepackage{delimset}
\def\given{\mskip 0.5mu plus 0.25mu\vert\mskip 0.5mu plus 0.15mu}
\newcounter{bracketlevel}%
\def\@bracketfactory#1#2#3#4#5#6{%
	\expandafter\def\csname#1\endcsname##1{%
		\global\advance\c@bracketlevel 1\relax%
		\global\expandafter\let\csname @middummy\alph{bracketlevel}\endcsname\given%
		\global\def\given{\mskip#5\csname#4\endcsname\vert\mskip#6}\csname#4l\endcsname#2##1\csname#4r\endcsname#3%
		\global\expandafter\let\expandafter\given\csname @middummy\alph{bracketlevel}\endcsname%
		\global\advance\c@bracketlevel -1\relax%
	}%
}
\def\bracketfactory#1#2#3{
\@bracketfactory{#1}{#2}{#3}{relax}{0.5mu plus 0.25mu}{0.5mu plus 0.15mu}
\@bracketfactory{b#1}{#2}{#3}{big}{1mu plus 0.25mu minus 0.25mu}{0.6mu plus 0.15mu minus 0.15mu}
\@bracketfactory{bb#1}{#2}{#3}{Big}{2.4mu plus 0.8mu minus 0.8mu}{1.8mu plus 0.6mu minus 0.6mu}
\@bracketfactory{bbb#1}{#2}{#3}{bigg}{3.2mu plus 1mu minus 1mu}{2.4mu plus 0.75mu minus 0.75mu}
\@bracketfactory{bbbb#1}{#2}{#3}{Bigg}{4mu plus 1mu minus 1mu}{3mu plus 0.75mu minus 0.75mu}
}%
\bracketfactory{clc}{\lbrace}{\rbrace}
\bracketfactory{clr}{\lparen}{\rparen}
\bracketfactory{cls}{\lbrack}{\rbrack}
\bracketfactory{dclr}{\lparen\!\lparen}{\rparen\!\rparen}
\bracketfactory{abs}{\lvert}{\rvert}
\bracketfactory{norm}{\lVert}{\rVert}
\bracketfactory{floor}{\lfloor}{\rfloor}
\bracketfactory{ceil}{\lceil}{\rceil}
\bracketfactory{angle}{\langle}{\rangle}

\let\original@left\left
\let\original@right\right
\renewcommand{\left}{\mathopen{}\mathclose\bgroup\original@left}
\renewcommand{\right}{\aftergroup\egroup\original@right}

\newcounter{ctr}\loop\stepcounter{ctr}\edef\X{\@Alph\c@ctr}%
	\expandafter\edef\csname s\X\endcsname{\noexpand\mathscr{\X}}
	\expandafter\edef\csname c\X\endcsname{\noexpand\mathcal{\X}}
	\expandafter\edef\csname b\X\endcsname{\noexpand\boldsymbol{\X}}
	\expandafter\edef\csname I\X\endcsname{\noexpand\mathbb{\X}}
\ifnum\thectr<26\repeat

\let\@IE\IE\let\IE\undefined
\newcommand{\IE}{\mathop{{}\@IE}\mathopen{}}
\let\@IP\IP\let\IP\undefined
\newcommand{\IP}{\mathop{{}\@IP}\mathopen{}}

\newcommand{\law}{\mathop{{}\sL}\mathopen{}}

\def\^#1{\relax\ifmmode {\mathaccent"705E #1} \else {\accent94 #1} \fi}
\def\~#1{\relax\ifmmode {\mathaccent"707E #1} \else {\accent"7E #1} \fi}
\def\*#1{\relax#1^\ast}
\edef\-#1{\relax\noexpand\ifmmode {\noexpand\bar{#1}} \noexpand\else \-#1\noexpand\fi}
\def\>#1{\vec{#1}}
\def\.#1{\dot{#1}}

\def\atop{\@@atop}

\def\hat#1{\widehat{#1}}

\renewcommand{\leq}{\leqslant}
\renewcommand{\geq}{\geqslant}
\renewcommand{\phi}{\varphi}
\newcommand{\eps}{\varepsilon}

\newcommand{\eq}{\eqref}
\newcommand{\I}{\mathop{{}\mathrm{I}}\mathopen{}}

\newcommand\indep{\protect\mathpalette{\protect\@indep}{\perp}}
\def\@indep#1#2{\mathrel{\rlap{$#1#2$}\mkern2mu{#1#2}}}

\newcommand{\toinf}{\to\infty}
\newcommand{\tozero}{\to0}
\newcommand{\longto}{\longrightarrow}
\newcommand{\wconv}{\Longrightarrow}

\def\dsub{d_{\mathrm{sub}}}
\def\dms{d_{\mathrm{ms}}}

\def\dcir{d^{\circ}}

\def\parsetime#1#2#3#4#5#6{#1#2:#3#4}
\def\parsedate#1:20#2#3#4#5#6#7#8+#9\empty{20#2#3-#4#5-#6#7 \parsetime #8}
\def\moddate{\expandafter\parsedate\pdffilemoddate{\jobname.tex}\empty}

\def\sep{\mathrm{sep}}

\crefname{equation}{}{}
\crefformat{equation}{\textup{(#2{#1}#3)}}
\crefrangeformat{equation}{\textup{(#3\textup{#1}#4)--(#5\textup{#2}#6)}}
\crefdefaultlabelformat{#2\textup{#1}#3}

\crefname{page}{p.}{pp.}

\newenvironment{equ}{\begin{equation}\begin{aligned}}{\end{aligned}\end{equation}}
\AfterEndEnvironment{equ}{\noindent\ignorespaces}

\newlist{condition}{enumerate}{10}
\setlist[condition]{label*=({A}\arabic*)}
\crefname{conditioni}{condition}{conditions}
\Crefname{conditioni}{Condition}{Conditions}

\crefname{subsection}{Subsection}{Subsections}

\DeclareMathOperator{\CM}{CM}
\def\1{\II}
\def\I{\II}

\DeclareMathOperator{\inj}{inj}
\DeclareMathOperator{\ind}{ind}
\DeclareMathOperator{\sub}{sub}

\DeclareMathOperator{\simple}{sim}

\def\Wtilde{\rlap{\raisebox{-0.2ex}{$\hspace{0.1ex}\widetilde{\phantom{\cW}}$}}\cW}
\def\hhat{\rlap{\raisebox{-0.2ex}{$\widehat{\phantom{h}}$}}h}
\makeatother

\begin{document}

\title{\sc\bf\large\MakeUppercase{Dense multigraphon-valued stochastic processes and edge-changing dynamics in the configuration model}}
\author{\sc Adrian R\"ollin and Zhuosong Zhang}

\date{\itshape National University of Singapore}

\maketitle

\begin{abstract} 
\noindent Time-evolving random graph models have appeared and have been studied in various fields of research over the past decades. However, the rigorous mathematical treatment of large graphs and their limits at the process-level is still in its infancy. In this article, we adapt the approach of \citet{Ath19} to the setting of multigraphs and multigraphons, introduced by \cite{Kol11}. We then generalise the work of \cite{Rath2012a} and \cite{Rat12}, who analysed edge-flipping dynamics on the configuration model --- in contrast to their work, we establish weak convergence at the process-level, and by allowing removal and addition of edges, these limits are non-deterministic.


\end{abstract}

\section{Introduction}

The mathematical theory of dense graphs and their limits, initiated in its modern form by \cite{Lov06a}, as well as its embellishments have become the focus of intense research over the past decade. While the case of dense simple graph sequences and their limits, called \emph{graphons}, is now well-understood, which is reflected in the survey articles of \cite{Bor08,Bor12} and in the book-length discussion of \cite{Lov12}, the case of dense multigraphs is much less developed. 

\cite{Kol11} adapted Lov\'asz and Szegedy's theory to dense multigraph sequences, and \cite{Rath2012a} and \cite{Rat12} illustrated the theory by applying it to determine the limits of dense configuration random multigraph models. The limits are called \emph{multigraphons} and are a natural extensions of graphons. The configuration model lends itself to such an analysis since in its basic version, it generally leads to multigraphs, and simple graphs can only be obtained either by conditioning or removal of multi-edges and loops. In the dense case, however, these operations distort the graph considerably and so the model is best analysed in its original multigraph form. 

Another line of research that has become increasingly important is that of network dynamics, since only rarely are networks static over time. However, the mathematical treatment of network dynamics is still not well developed, despite a rather large literature on such models.
\cite{Erdos1960a} analysed growing random graphs, \cite{Holland1977} looked at the evolution of social networks, and accounts of subsequent developments are given by \cite{Snijders2001} and \cite{Snijders2010a} with a more statistical perspective. Recently, results have been appearing  more frequently in mathematical literature, too, such as those of \cite{Basak2015} and \cite{Basu2017} to name a few.  \cite{Rath2012a} and \cite{Rat12} in
fact also considered edge-flipping dynamics of the configuration model.  In the context of graph limits, \cite{Cra16a} was the first
to develop a cohesive stochastic-process point of view, and he introduced and studied graphon-valued processes mainly through the lens of the theory of \cite{Ald81a} and  \cite{Hoo89}, which was shown to be equivalent to theory of \cite{Lov06a}; see \cite{Dia07}. Another, more direct approach was taken by \citet{Ath19}, who established a weak limit theory for graph-valued stochastic processes with graphon-valued process limits, but many questions remain open, such as how to describe generators of graphon-valued Markov processes.

The aim of the present article is to develop a weak limit theory for multigraphon-valued stochastic processes analogous to that of \citet{Ath19}. This is done in essence by defining the Skorohod topology on the space of c\`adl\`ag multigraphon-valued paths. In order to achieve this, we
introduce a new metric and show that this metric makes the space of multigraphons (or rather the quotient space under measure-preserving transformations) complete and separable, which is an important ingredient in the context of process-level analysis. We also construct and study a class of multigraph-valued processes which give rise to these limits, and our
workhorse will be the configuration model with dynamics defined through by flipping, deleting and adding edges. This extends the results of \cite{Rat12} in one key point: Our limiting processes are truly stochastic, that is, not deterministic. We also highlight that, to the best of our knowledge, this is the first example of stochastic process level convergence of time-evolving graphs where the network structure is not the direct consequence of an underlying, well-understood stochastic process (such as the Moran model used by \citet{Ath19}), but emerges purely due to local edge manipulations.

The rest of this paper is organized as follows. In Section \ref{sec1a}, we give a brief overview of the space of  multigraphons and its quotient space under measure-preserving transformations, define a new metric on the quotient space and establish completeness and separability. We then provide characterisations of weak convergence of multigraphon-valued stochastic processes similar to those of \citet{Ath19}. In Section \ref{sec1}, we first discuss the configuration model and show the basic convergence to its multigraphon limit, and then introduce the edge-flipping dynamics and establish process-level convergence.

\section{Multigraphon-valued stochastic processes} \label{sec1a}

\subsection{Multigraphs and multigraphons}%
\label{sub1}

In this article, by \emph{multigraph}, we mean a graph $G$ on a vertex set $V(G)$, where we allow for multiple edges and multiple loops. We loosely follow the setup of \cite{Kol11}, and represent a multigraph $G$ by its adjacency matrix $(z_{ij})_{i, j \in [n]}$, where $z_{ij}$
equals the number of edges connecting the vertices labelled by $i$ and $j$ if $i \neq j$, and where it equals two times the number of loops of vertex $i$ if $i=j$. 
Let $v(G)$ be the number of vertices, let $e(G) = \sum_{1 \leq i < j \leq v(G)} z_{ij}$ be the number of non-loop edges, and let $l(G) = \sum_{i = 1}^{v(G)} z_{ii}/2$ be the number of loops in $G$. 
For $n \in \IN,$ let $\cM_n$ be the set of multigraphs on $[n]$ and let $\cM = \cup_{n = 1}^{\infty} \cM_n$. If $G_1
\in
\cM_n$ and $G_2 \in \cM_n$, we denote by $G_1 + G_2$ the multigraph on $[n]$ whose adjacency matrix is the sum of adjacency matrices of $G_1$ and $G_2$.

In order to define the distance between two multigraphs, we follow the paper of \cite{Kol11}, and define the subgraph density functionals as follows. Let $k\geq 1$ and $n \geq 1$, and let $F = (a_{ij})_{i,j \in [k]}\in \cM_k$ and $G = (z_{ij})_{i, j \in [n]}\in\cM_n$; then, define the \emph{homomorphism density of $F$ in $G$} as 
\be{
		t_F(G)  =  \frac{1}{n^k} \sum_{\sigma \colon [k] \to [n]}  \I \bcls{\forall i, j \in [k] : a_{ij} \leq z_{\sigma(i) \sigma(j)}}, 
}
where  the summation $\sum_{\sigma \colon [k] \to [n]}$ ranges over all maps $\sigma$ from $[k]$ to $[n]$.
For finite multigraphs, it is more convenient to work with injective homomorphism densities and induced homomorphism densities, which are both equivalent forms of homomorphism densities. Let $F = (a_{ij})_{i, j \in [k]}	 \in \cM_k$ and $G = (z_{ij})_{i, j \in [n]} \in \cM_n$, define  
\ben{ \label{eq2-12}
	t^{\inj}_F(G)  =  
		\frac{1}{(n)_k} \sum_{\sigma \colon [k] \hookrightarrow [n]} \1 \bigl\{ \forall i, j \in [k] : a_{ij} \leq z_{\sigma(i) \sigma(j)} \bigr\}
}
if $k\leq n$, and $t^{\inj}_F(G)=0$ otherwise; here, the summation is over all injective maps $\sigma$ from $[k]$ to $[n]$ and where $(n)_k = n(n-1) \cdots (n - k + 1)$ is the falling factorial. Similarly, define
\ben{  \label{eq2-13}
	t_F^{\ind }(G) 
	= 
		\frac{1}{(n)_k} \sum_{\sigma \colon [k] \hookrightarrow [n]} \1 \bigl\{ \forall i, j \in [k] : a_{ij} = z_{\sigma(i) \sigma(j) }\bigr\}      
}
if $k\leq n$ and $t_F^{\ind }(G) = 0$ otherwise.
 By a standard inclusion–exclusion argument, 
\be{
	\bigl\lvert t_{F}^{\inj} (G) - t_F (G) \bigr\rvert \leq \frac{1}{v(G)} \binom{v(F)}{2}  .
}
Note that $\cM$ is countable. In order to define an appropriate distance between multigraphs, consider the map $\tau: \cM \to [0, 1]^{\cM}$ defined as 
\be{
	\tau(G) \coloneqq \bigl( t_{F}(G) \bigr)_{F \in \cM} \in [0,1]^{\cM}.
}%
Since $[0,1]^{\cM}$ is a compact space (equipped with the canonical metric), it would be tempting to take closure of the image of $\tau(\cM)$, which would then also be compact; see discussion of \cite[p.~7]{Dia07}. However, there is no guarantee that the closure
has a nice representation, as happens to be the case for simple graphons. Indeed, if $K_n$ denotes a graph on $n$ vertices with $n$ edges between every pair of vertices, we have $t_F(K_n)\to1$ as $n\toinf$ for every $F\in\cM$, but the limiting element
$(1)_{F\in\cM}\in[0,1]^{\cM}$ does not have a multigraphon representation (see Definition~\ref{def1} below). However, we do not need compactness of the underlying metric space — completeness and separability will suffice to develop a suitable theory. 
 
To this end, we define the \emph{multisubgraph distance $\dms$} between two multigraphs  $G_1, G_2 \in \cM$ as 
\besn{
	\label{eq2-1}
	\dms (G_1, G_2) 
	& = \sum_{i = 1}^{\infty} 2^{-i} \bigl\lvert t_{F^*_i} (G_1) - t_{F^*_i}(G_2) \bigr\rvert  \\
	& \quad  +  \sum_{r \geq 0} \lvert t_{K_{2,r}}^{\ind}(G) - t_{K_{2,r}}^{\ind}(G) \rvert  \\
	& \quad  + \sum_{ r \geq 0 }\lvert t_{L_r}^{\ind}(G) - t_{L_r}^{\ind}(G) \rvert, 
}%
where $F^*_1, F^*_2, \dots$ is some enumeration of all multigraphs, where 
$K_{2,r}$ is the graph on two vertices with $r$ edges
connecting them, and where $L_{r}$ is the graph on one vertex with $r$ loops.
Note that  for different orderings of $F^*_1, F^*_2, \dots $, the subgraph distances are equivalent.

In order to define the completion of $\cM$ with respect to the distance $\dms$, we introduce multigraphons. 
For $j = 1, 2$, let $L_1( [0,1]^j )$ be a space of Lesbegue integrable functions $\phi: [0,1]^j \to \IR$, where functions which agree almost everywhere with respect to the $j$-dimensional Lebesgue measure  are identified as one object.
\begin{definition}\label{def1}
We say $h: \mathbb{N}_0 \times [0,1]^2 \to [0,1]$ is a \emph{multigraphon} if 
\begin{enumerate}[$(i)$]
	\item  for each $r \geq 0$, the function $(x,y) \mapsto h(r; x, y)$ belongs to $L_1([0,1]^2)$ and the function $x \mapsto h(r; x, x)$ belongs to $L_1([0,1])$; 
	\item for any $r \geq 0$ and for $(x, y) \in [0,1]^2$, 
\ben{
	 h(r; x, y) = h(r; y, x) , \quad 
			\sum_{r=0 }^{\infty} h(r ; x, y)  = 1 ,
	\label{111}
}
and for $x \in [0,1]$,  
\ben{\label{112}
	h(2 r +1 ; x, x) = 0. 
}%
\end{enumerate}
\end{definition}

For any two multigraphons $h_1$ and $h_2$, we write $h_1 \equiv h_2$ if for all $r \geq 0$,
\begin{align*}
	 \int_{[0,1]^2} | h_1(r; x, y) - h_2(r; x, y) | dx dy & = 0,\\
	 \int_{[0,1]} | h_1(r; x, x) - h_2(r; x, x) |dx & =  0. 
\end{align*}
Let $\cH$ be the class of all equivalent classes of multigraphons with respect to ``$\equiv$''. Let $h \in \cH$; while strictly speaking, $h$ is an equivalence class of multigraphons, we will always interpret $h$ as a \emph{representative} of the corresponding equivalence class, that is, as an actual multigraphon,  without making a notational distinction between the two. But the reader needs to keep in mind that statements about $\cH$ are to be understood as statements about the respective equivalence classes. 

For each $h \in \cH$ and $F = (a_{ij})_{i , j \in [k]} \in \cM_k$, define the \emph{homomorphism density of $F$ in $h$} as
\begin{equ}
		t_F(h)  =  \int_{[0,1]^k} \prod_{1 \leq i \leq j \leq k} \sum_{r = a_{ij}}^{\infty} h(r ; x_i, x_j ) \,dx_1\dots dx_k.
	    \label{eq2-2}
\end{equ}
Similarly, define the \emph{induced homomorphism density of $F$ in $h$} as
\be{
	t_F^{\ind}(h) = \int_{[0,1]^k} \prod_{1 \leq i \leq j \leq k} h(a_{ij} ; x_i, x_j ) \,dx_1\dots dx_k .
    \label{eq2-3}
}%
Alternatively, if $U_1, \dots, U_k$ are independent random variables, distributed uniformly on $[0, 1]$, we can write 
\besn{ \label{eq2-4}
		t_F(h) & =  \IE \biggl\{ \prod_{1 \leq i \leq j \leq k} \sum_{r = a_{ij}}^{\infty} h(r; U_i, U_j) \biggr\}, \\
		t_F^{\ind }(h) & =  \IE \biggl\{ \prod_{1 \leq i \leq j \leq k} h (a_{ij} ; U_i, U_j) \biggr\}.
}%
Moreover, if $F_1$ is isomorphic to $F_2$, then $t_{F_1}(h) = t_{F_2}(h)$ and $t_{F_1}^{\ind}(h) = t_{F_2}^{\ind}(h)$. Similarly as for multigraphs, we define $\dms$ for multigraphons as
\besn{
	\label{eq2-6}
	d_{\mathrm{ms}}(h,h')
	& = \sum_{i\geq 1}2^{-i}\babs{t_{F_i^*}(h)-t_{F_i^*}(h')} \\
	& \quad + \sum_{r\geq 0}\babs{t^{\ind}_{K_{2,r}} (h) -t^{\ind}_{K_{2,r}}(h')} \\
	& \quad + \sum_{r\geq 0}\babs{t^{\ind}_{L_r}(h)-t^{\ind}_{L_r}(h')}, 
}%
Note that the second and third sums in \cref{eq2-6} are always finite due to the condition that $\sum_{r\geq 0} h(r;x,y) = 1$.

We can embed the space of multigraphs in the space of multigraphons in the usual manner: For any multigraph $G = (z_{ij})_{i,j \in [n]} \in \cM_n$, let the corresponding multigraphon $h^{G}$ be defined as 
\begin{align*}
	h^{G} (r; x,y) = \1 \bigl[ z_{\lceil nx \rceil \lceil ny \rceil } = r \bigr] , 
	\quad
	k \geq 0  .
\end{align*}
\cite{Kol11} showed that $t_F (G) = t_F (h^{G})$ for any $F \in \cM$;
this justifies defining $\dms$ between a multigraph and a multigraphon~as 
\begin{align*}
	\dms (G, h) = \dms (h^{G}, h).
\end{align*}
Note that $\dms$ is only a pseudo-metric; that is, $d_{\mathrm{ms}}(h,h')$ may be zero, even though $h$ and $h'$ are not equal almost everywhere. This happens if $h$ and $h'$ are related via measure-preserving transformations, which is analogous to the graphon case. We will discuss this later.

The distance $\dms$ is novel in two ways. First, although multigraphon and its subgraph density functionals were introduced by \cite{Kol11} and further discussed by \cite{Rat12}, distances on the multigraphon space have not yet been defined and analysed to the best of our knowledge. Second, the metric $\dms$ is not a naive generalization of the subgraph distance and cut distance for simple graphon space (c.f. \cite{Lov06a}), because compared to the subgraph distance for simple graphons, there are two additional terms involved in ${\dms}$, which is what ensures the completeness property of the space $(\cH, \dms)$.    
\begin{lemma}
	\label{lem0}
	The pseudo-metric space $(\cH, \dms)$ is complete and separable. 
\end{lemma}
\begin{proof}
We first prove that $(\cH, \dms)$ is complete. To this end, let $h_1, h_2, \dots$ be a Cauchy sequence in $(\cH, \dms)$. By the first sum in the definition of $\dms$, it follows that, for any $F \in \cM$, $(t_F(h_n))_{n \geq 1}$ is also a Cauchy sequence. Hence, $\lim_{n\toinf} t_F(h_n)$ exists. 
Define the function $f: \cM \to [0,1]$ as $f(F) = \lim_{n \toinf} t_{F}(h_n)$. 
We proceed in two steps: We first prove that there exists a multigraphon $h \in \cH$ such that $t_F(h) = f(F)$ for all $F \in \cM$; then, we prove that $\dms(h_n , h) \to 0$ as $n \toinf$. 

For the first step, we need to prove that $f$ is non-defective; that is, we need to show that for any $k\geq 1$  and any sequence $F_1, F_2, \ldots \in \cM_k$ with $\lim_{j\toinf} ({ e(F_j)+l(F_j)}) = \infty$, it follows that  $\lim_{j\toinf} f(F_j) = 0$. 

Recall that  $K_{2, j}$ denotes the multigraph on two vertices with $j$ multiple edges and that $L_j$ denotes the multigraph on one vertex with $j$ loops.
As $(h_n)_{n \geq 1}$ is a Cauchy sequence in $(\cH, \dms)$, we have that for any $\eps > 0$, there exists $n_0 \coloneqq n_0(\eps)$ such that 
\begin{multline}\label{eqbb}
	\sum_{r \geq 0} \bigl( | t^{\ind}_{K_{2,r}}(h_n)  - t^{\ind}_{K_{2,r}}(h_{n_0})| + | t^{\ind}_{L_{r}}(h_n)  - t^{\ind}_{L_{r}}(h_{n_0})| \bigr) \leq \eps/2 
	\\
	\text{for all $n \geq n_0$.}
\end{multline}%
For this $n_0$, as $h_{n_0} \in \cH$, by \cref{111}, there exists $r_0 \coloneqq r_0(n_0, \eps)$ such that 
\ben{
	\sum_{ r \geq r_0 } \bigl( t^{\ind}_{ K_{2,r} }(h_{n_0}) + t^{\ind}_{L_r}(h_{n_0}) \bigr) \leq \eps/2, 
	\label{eqaa}
}%
By \cref{eqbb,eqaa}, we have for all $n \geq n_0$, 
\begin{equ}
	\label{eqcc}
	\sum_{ r \geq r_0 } \bigl( t^{\ind}_{ K_{2,r} }(h_{n}) + t^{\ind}_{L_r}(h_{n}) \bigr) \leq \eps. 
\end{equ}
Since $\lim_{j \toinf} ( e(F_j) + l(F_j) ) \to \infty$, there exists $j_0 \coloneqq j_0(r_0, k) > 1$ such that $e(F_j) + l(F_j) \geq k^2 r_0$ for all $j \geq j_0$. 
Now, for each $j \geq j_0$, at least one of the following two statements must be true: 
\begin{enumerate}[(a)]
	\item $F_j$ contains a vertex with $r_0$ loops; 
	\item $F_j$ containts a pair of vertices with $r_0$ multiple edges between them. 
\end{enumerate}
By \cref{eqcc}, for any $n \geq n_0$ and $j \geq j_0$ (both $n_0$ and $j_0$ depend only on $\eps$),  
\bs{	\label{eq2-16}
	t_{F_j} (h_{n}) & \leq \max \{ t_{K_{2, r_0}}(h_{n}) , t_{L_{r_0}} (h_{n}) \}\\
  & \leq \sum_{r \geq r_0} \bigl( t^{\ind}_{ K_{2,r} }(h_{n}) + t^{\ind}_{L_r}(h_{n}) \bigr) \leq \eps.       
}%
Letting $n\toinf$ in \cref{eq2-16}, we have 
\begin{align*}
	f(F_j) = \lim_{n \toinf} t_{F_j}(h_n) \leq \eps \quad \text{for all $j \geq j_0$}. 
\end{align*}
Noting that $f(F_j) \geq 0$ for all $j \geq 1$, we then conclude that $f(F_j) \to 0$ as $j \to \infty$, 
which implies by definition that $f$ is non-defective. By \cite[Theorem\;1]{Kol11}, we conclude that there exists a multigraphon $h \in \cH$ such that 
\be{
	f(F) = t_F(h) \text{ for all } F \in\cM. 
}%
This concludes the first step, and it remains to show that $\dms(h_n, h) \to 0$ as $n \to \infty$ as a second step.
As $t_F(h_n) \to t_F(h)$ for all $F \in \mathcal{M}$, by \cite[Lemma\;1]{Kol11}, we have $t_F^{\ind} (h_n) \to t^{\ind}_F(h)$ for all $F \in \{ K_{2, r}, L_{r} : r = 0, 1, \dots \}$. Thus, it follows that for all $r \geq 0$, 
\begin{align}
	t^{\ind}_{K_{2,r}}(h_n) \to t^{\ind}_{K_{2,r}} (h),\qquad 
	t^{\ind}_{L_{r}}(h_n) \to t^{\ind}_{L_{r}} (h) .
      \label{eq2-20}
\end{align}
Recalling that $h_n , h \in \cH$, and hence, by \cref{111}, we have 
\begin{align*}
	\sum_{r \geq 0} t^{\ind}_{K_{2,r}}(h_n) = 
	\sum_{r \geq 0} t^{\ind}_{K_{2,r}}(h) = \sum_{r \geq 0} t^{\ind}_{L_{r}}(h_n) = \sum_{r \geq 0} t^{\ind}_{L_{r}}(h_n) = 1. 
\end{align*}
By \cref{eq2-20} and the dominated convergence theorem, we have as $n \toinf$, 
\begin{equ}
	\sum_{r \geq 0} \lvert t^{\ind}_{K_{2,r}}(h_n) - t^{\ind}_{K_{2,r}}(h) \rvert \to 0,\qquad  \sum_{r \geq 0} \lvert t^{\ind}_{L_{r}}(h_n) - t^{\ind}_{L_{r}}(h) \rvert \to 0.
    \label{eq2-21}
\end{equ}
Recalling the fact that $t_F(h_n) \to t_F(h)$ for every $F \in \mathcal{M}$ together with \cref{eq2-21}, it is now routine to conclude that $\dms(h_n , h) \to 0$.

Now, we move to prove the separability of $(\cH, \dms)$ by 
showing that there exists a countable subset $\cH^\sep\subset\cH$ with the property that, for every $h\in\cH$, there is a sequence $h_1,h_2,\ldots\in\cH^\sep$ such that $\dms(h_n,h)\to0$. The latter is
implied if we can show that $t_F(h_n)\to t_F(h)$ (which in particular implies that $t^{\ind}_{K_{2,r}}(h_n)\to t^{\ind}_{K_{2,r}}(h)$ and $t^{\ind}_{L_{r}}(h_n)\to t^{\ind}_{L_{r}}(h)$).

\def\dg{\mathrm{dg}}
\def\sq{\mathrm{sq}}
We first introduce some notation.  
Recall that for $j = 1, 2$,  $L_1([0,1]^j)$ is the space of functions $\phi: [0,1]^j \to \IR$ such that $|\phi|$ is Lebesgue integrable, where functions which agree almost everywhere are identified.
For $\phi \in L_1([0,1]^2)$, let $\phi_{\dg}(x) = \phi(x, x)$. 
Let  
\begin{align*}
	{\cG} = \{ \phi \in L_1 ( [0,1]^2 ):  \phi \geq 0, \phi_{\dg} \in L_1( [0,1] ) \text{ and } \phi(x,y) = \phi(y,x)  \}. 
\end{align*}
For $\phi, \phi' \in \cG$, we introduce the metric 
\be{
	d_1 (\phi, \phi')  = d_{\sq} (\phi, \phi')  + d_{\dg} (\phi, \phi') , 
}%
where 
\begin{equ}
	d_{\sq} ( \phi, \phi' ) & =  \int_{ [0,1]^2 }  \bigl\lvert \phi(x, y) - \phi'(x, y) \bigr\rvert dx dy   , \\
	d_{\dg} (\phi, \phi') & = 
	\int_{ [0,1] } \lvert \phi(x, x) - \phi'(x, x) \rvert dx. 
    \label{eq:box}
\end{equ}
It is routine to show that $(\cG, d_1)$ is a metric space.

Next, we prove that $(\cG,d_1)$ is separable. Recall that the metric space $L_1( [0,1]^2 )$ is separable. As $(\cG, d_{\sq})$ is a subspace of $L_1 ( [0,1]^2 )$, it is also separable, since every subspace of a separable metric space is again separable. Let $\cU_0$ be a countable and dense subset of $(\cG, d_{\sq})$. 
Let $\cG_{ \dg } = \{ f\in L_1( [0,1] ): f \geq 0\}$. 
By a similar argument, we have the space $\cG_{ \dg }$ contains a dense countable subset $\cV_0$. 
Let 
\begin{multline*}
	\cG^{\sep} 
	= \{ \phi \in \cG : \text{$\exists\, U \in \cU_0$ and $f \in \cV_0$ such that}\\
		\text{$\phi(x,y) = U(x,y)$ a.e. for $(x,y) \in [0,1]^2$ with $x\neq y$,} \\
		\text{and $\phi(x, x) = f(x)$ a.e.\ for $x \in [0,1]$} \}. 
\end{multline*}
Since $\cU_0 \times \cV_0$ is countable, it follows that $\cG_{\sep}$ is also countable. Moreover, by the definition of $d_1$ and by the dense properties of $\cU_0$ and $\cV_0$, we have $\cG^{\sep}$ is  also dense in $\cG$ with respect to $d_1$. This implies that $(\cG, d_1)$
is separable. 

For $m \geq 0$, let 
\begin{multline*}
	\cG_m = \{ g = ( g(0), g(1), \dots ) \in \cG^{\IN} : \\
	g(r) \in \cG \text{ for $0 \leq r \leq m$ and $g(r) \equiv 0$ for $r > m$}\}. 
\end{multline*}
For any $g\in \cG^{\IN}$ and $r \geq 0$, 
let $g^{\geq r} : [0,1]^2 \to [0,\infty)$ be defined as 
\begin{equ}
	g^{\geq r} (x, y) = \sum_{s = r}^{\infty} g(s; x, y).
    \label{eq-hsqr}
\end{equ}
Note that $g^{\geq r} \in \cG$.
For any $m \geq 1$, we equip the space $\cG_m$ with the distance 
\begin{equ}
	d_{2} (g_1, g_2) = \sup_{r \geq 0} d_1 ( g_1^{\geq r},  g_2^{\geq r}) \quad \text{for } g_1 , g_2 \in \cG_m.  
    \label{eq-dotimes}
\end{equ}
Again, we have for each $m \geq 0$, $(\cG_m, d_{2})$ is a metric space.

We then move on to prove separability of $(\cG_m, d_{2})$ for every finite $m \geq 0$.
To this end, let 
\begin{align*}
	\cG_m^{\sep} = \{ g \in \cG_m : g(r) \in \cG^{\sep} \text{ for $0 \leq r \leq m$ and $g(r) \equiv 0$ for $r > m$} \}. 
\end{align*}
Thus, $\cG_m^{\sep}$ is countable. Now, we prove $\cG_m^{\sep}$ is dense in $(\cG_m, d_{2})$.
For any $g \in \cG_m$, we have $g(r) \in \cG$ for $0 \leq r \leq m$ and $g(r) = 0$ for $r > m$. By separability of $(\cG, d_1)$, for any $n \geq 1$ and $r \geq 0$, there exists a sequence $(\psi_{r,M})_{M \geq 1} \subset \cG^{\sep}$ that converges to $g(r)$. Then, there exists a number $M(r, n, m)$ such that 
\begin{align*}
	d_1( \psi_{r,M(r,n, m)}, g(r) ) < \frac{1}{2^{r + 1} (m+1) n}.  
\end{align*}
Let $g_n \in \cG^{\IN}$ be defined as $g_n(r) = \psi_{r,M(r,n, m)}$ for $0 \leq r \leq m$ and $g_n(r) = 0$ for $r > m$. Then, we have $g_n \in \cG^{\sep}_m$ and 
\begin{equation*}
	\begin{split}
		d_{2} (g_n, g) & =  \sup_{r \geq 0} d_1 ( g_n^{\geq r}, g^{\geq r} )
							 \leq \sum_{r = 0}^\infty \sum_{s = r}^m d_1 (g_n(s), g(s)) \\
							 & = \sum_{r = 0}^\infty \sum_{s = r}^m d_1 (\psi_{s, M(s, n, m)}, g(s)) 
							 \leq \frac{1}{n}. 
	\end{split}
\end{equation*}
Thus, $(g_n)_{n \geq 1}$ converges to $g$ with respect to $d_{2}$. This shows that $\cG_m^{\sep}$ is dense in $(\cG_m, d_{2})$, and hence, $(\cG_m, d_{2})$ is separable.  

We are now ready to construct $\cH^{\sep}$. For $m \geq 0$, let 
\begin{multline*}
	\cH_m = \biggl\{ h \in \cG_m : \sum_{r \geq 0} h(r) \equiv 1,\, h_{\dg}(2r+1) \equiv 0 \text{ for all $r \geq 0$}\biggr\}. 
\end{multline*}
For each $m \geq 0$, we have $(\cH_m, d_{2})$ is a subspace of the metric space $(\cG_m, d_{2})$, and thus, is also separable. 
Let $\cH^{\sep}_m$ be a countable and dense subset of $(\cH_m, d_{2})$, and $\cH^{\sep} = \cup_{m \geq 0} \cH^{\sep}_m$. 
Thus, $\cH^{\sep}$ is countable.
Moreover, we have $\cH^{\sep} \subset \cH$. 

We finish this proof by showing that for any $h \in \cH$, there exists a sequence $(h_n)_{n \geq 1} \subset \cH^{\sep}$ such that for any $F \in \cM$, $\lvert t_F(h_n) - t_F(h) \rvert \to 0$ as $n \to \infty$. To this end, fix $h \in
\cH$.

For each $n \geq 1$, there exists~$m(n) > 0$ such that  
\begin{equ}
	\sum_{r \geq m(n)} \int_{ [0,1]^2 }h(r; x, y) \, dx dy  + \sum_{r \geq m(n)}  \int_{ [0,1] } h(r; x, x) dx  < \frac{1}{n}. 
    \label{eq2-hbound}
\end{equ}
Let $\bar h_n \in \cH$ be defined as 
\begin{align}
	\bar{h}_n (r) = 
	\begin{cases}
		h(r) & \text{ if $0 \leq r < m(n)$},\\
		\sum_{s \geq m(n)} h(s) & \text{ if $r = m(n)$,}\\
		0 & \text{ otherwise. }
	\end{cases}
	\label{eq-hbarn}
\end{align}
Thus, we have $\bar{h}_n \in \cH_{m(n)}$ and $\bar{h}_n^{\geq r} = h^{\geq r} $ for $0 \leq r \leq m(n)$. By the separability of $\cH_{m(n)}$,  there exists a sequence $( h_M^{\sep})_{M \geq 1} \subset \cH^{\sep}_{m(n)} \subset \cH^{\sep}$ such that $d_2(h_M^{\sep},\bar h_n) \to 0$ as $M \to \infty$. Therefore, there exists an $M(n) > 0$ such that
\begin{equation}
	d_{2}(h^{\sep}_{M(n)}, \bar{h}_n) \leq
	1/n. 
	\label{eq23}
\end{equation}

Choose $h_n = h^{\sep}_{M(n)}$. 
Now, it suffices to show that $| t_F(h_n) - t_F(h) | \to 0$ for all $F \in \cM$ as $n \to \infty$. 
Let $k \geq 1$ and $F = (a_{ij})_{i,j \in [k]} \in \mathcal{M}_k$ be arbitrary. If $\max_{i,j} a_{ij} > m(n)$, by \cref{eq2-hbound} and by definition of the multigraph parameter $t_F$, it is easy to see that $t_F(h_n) = 0$ and $t_F(h) \leq k^2/n$, and hence that
\begin{equ}
	 \abs{t_F(h_n) - t_F(h)}\leq k^2/n.
    \label{eq-tF0}
\end{equ}
Moreover, by \cref{eq-dotimes,eq-hbarn,eq23}, it follows that 
\bmn{
	\sup_{r \geq 0} d_1 ( h_n^{\geq r}, {h}^{\geq r} ) = \sup_{r \geq 0} d_1 ( h_n^{\geq r}, \bar{h}_n^{\geq r} ) = d_{2}(h^{\sep}_{M(n)}, \bar{h}_n)  \leq \frac{1}{n}\\ \text{ for $0 \leq r \leq m(n)$. }
	\label{eq2-24}
}%
If $\max_{i,j} a_{ij} \leq m(n)$, we apply Lemma~\ref{lem:10} (see below) and \eq{eq2-24}, and obtain
\bes{
  |t_F(h_n) - t_F(h)| & \leq 
  \sum_{ 1 \leq i < j \leq k } d_{\sq} (h_n^{[a_{ij}]}, h^{[a_{ij}]}) + \sum_{1\leq i \leq k} d_{\dg} ( h_n^{[a_{ii}]}, h^{[a_{ii}]} ) \\
  & \leq \frac{k(k-1)}{n}+\frac{k}{n}.
}%
Hence, $\abs{t_F(h_n) - t_F(h)}\tozero$ as required.
\end{proof}

\begin{lemma}\label{lem:10} 
\def\sq{\mathrm{sq}}
\def\dg{\mathrm{dg}}
	Let $h_1, h_2\in\cH$, let $k\geq 1$, and let $F \in \cM_k$. Let $h^{\geq r}$ be defined as in \cref{eq-hsqr}. Then
\be{
  \babs{t_F(h_1)-t_F(h_2)}\leq \sum_{ 1 \leq i < j \leq k } d_{\sq} (h_1^{[a_{ij}]}, h_2^{[a_{ij}]}) + \sum_{1\leq i \leq k} d_{\dg} ( h_1^{[a_{ii}]}, h_2^{[a_{ii}]} ).
}%
\end{lemma}
\begin{proof}
\def\sq{\mathrm{sq}}
\def\dg{\mathrm{dg}}
Let 
\begin{align*}
	\theta (u) = \int_{ [0,1]^k }\prod_{1 \leq i \leq j \leq k} &  \biggl( u \sum_{r \geq a_{ij}} h_1(r; x_i, x_j) \\
	& \quad + (1 - u) \sum_{r \geq a_{ij}} h_2(r; x_i,x_j) \biggr) dx_1 \dots d x_k, 
\end{align*}
and thus 
\begin{align*}
	\theta'(u) & =  \int_{ [0,1]^k } \sum_{1 \leq i \leq j \leq k} b_{ij}(u)   \bbbclr{ \sum_{r \geq a_{ij}} \bclr{h_1(r; x_{i}, x_{j}) - h_2(r; x_{i}, x_{j})}  } dx_1 \dots d x_k.
\end{align*}
where 
\be{
	b_{ij} (u) =
	\begin{dcases*}
		\prod_{1 \leq i' \leq j' \leq k \atop (i',j') \neq (i,j)} c_{i'j'}(u)	 & if $\displaystyle \sum_{r \geq a_{ij}} (h_1(r; x_i, x_j) \neq \sum_{r \geq a_{ij}} h_2(r; x_i, x_j))$,\\
		0 & otherwise,
	\end{dcases*}
}%
and where 
\begin{align*}
	c_{ij}(u) = u \sum_{r \geq a_{ij}} h_1(r; x_{i}, x_{j}) + (1 - u) \sum_{r \geq
			a_{ij}} h_2(r; x_{i},x_{j}).
\end{align*}
It follows that $0 \leq b_{ij}(u) \leq 1$ for all $0 \leq i \leq j \leq k$ and $u \in [0,1]$. 
Hence, for all $u \in [0,1]$, 
\begin{align*}
	\abs{ \theta'(u) } 
	  & \leq \sum_{1 \leq i < j \leq k}  
	  \int_{ [0,1]^2 }\biggl\lvert\sum_{r \geq a_{ij}} \bigl( h_1(r; x, y) - h_2(r; x, y) \bigr) \biggr\rvert dx dy   \\
	  & \quad + \sum_{i = 1}^k   \int_{ [0,1] } \biggl\lvert \sum_{r \geq a_{ii}} \bigl( h_1(r; x, x) - h_2(r; x, x) \bigr) \biggr\rvert dx  .
\end{align*}
The claim now easily follows.
\end{proof}

The collection of maps $(t_{F^*_i})_{i\geq 1}$ is not injective, since the values $(t_{F^*_i}(h))_{i\geq 1}$ determine $h\in\cH$ only up to a measure-preserving transformation. Hence, we proceed to define an equivalence relation ``$\cong$'' in the canonical way.	Let $h_1, h_2 \in \cH$; we say $h_1$ and $h_2$ are equivalent and write $h_1 \cong h_2$ if $t_F(h_1) = t_F(h_2)$ for all $F \in \cM$. Observing that
\begin{align*}
	t_{K_{2,r}}^{\ind}(h) = t_{K_{2,r}}(h) - t_{K_{2,r + 1}}(h), \quad t_{L_r}^{\ind}(h) = t_{L_r}(h) - t_{L_{r + 1}}(h) , 
\end{align*}
and using \cref{111,112} to represent the second and third sum of in \eq{eq2-6}, it easily follows that $h_1 \cong h_2$ if and only if $\dms(h_1, h_2) = 0$. As an immediate consequence of \cref{lem0}, we have the following result.

\begin{corollary}
	The metric space $( \mathcal{H} \setminus \!\cong, \dms)$ is complete and separable. 
\end{corollary}

\begin{remark}
	Since functions that are continuous on the pseudo-metric space $(\mathcal{H}, \dms)$ are also continuous on the induced metric space $(\mathcal{H} \setminus \!\!\cong, \dms)$ and vice versa, there is no need to distinguish the two spaces as far as weak convergence is concerned, since weak convergence is determined by continuous and bounded functions. Therefore, in what follows, we will not distinguish between
	the pseudo-metric space $(\mathcal{H},\dms)$ and the metric space $(\mathcal{H} \setminus \!\cong, \dms)$ and simply use the notation $(\mathcal{H},\dms)$ throughout.
\end{remark}

\subsection{Simple graphons}

We now discuss some relations between multigraphons and simple graphons.
First, let $h$ be a multigraphon, and for fixed $r \geq 0$, we note that $h(r; \cdot , \cdot) : [0,1]^2 \to [0,1]$ is a simple graphon. 
Second, a simple graphon is a special case of multigraphon — for any graphon $\hhat$, we can define its corresponding multigraphon as 
\ben{
	h(0; x, y) = 1 - \hhat (x, y), \quad h(1; x, y) = \hhat (x, y) , \quad h(r; x, y) = 0  \text{ for $r \geq 2$. }
    \label{eq2-9}
}%
Recall now that the homomorphism density of any simple graph $F$ on $k$ vertices in a simple graphon $\hhat$ is defined as
\be{
	t^{\simple}_F(\hhat) =  \int_{[0,1]^k} \prod_{ij \in F} \hhat( x_i, x_j ) d x_1 \dots d x_k, 
}%
where $ij \in F$ indicates that $(i,j)$ is an edge in $F$; see \cite{Lov06a}.
It is easy to see that $t^{\simple}_F (\hat h) = t_F (h)$, where $h$ is defined as in \cref{eq2-9}.

\subsection{Weak convergence for multigraphon-valued random elements}%
\label{sub2.2}
In what follows, we use ``$\longrightarrow$'' to denote the convergence with respect to the underlying (pseudo)metric space, and we use ``$\wconv$'' to denote weak convergence, defined in the usual way. Specifically, in the space $(\cH, \dms)$, we say that a sequence of equivalence classes of multigraphons $(h_n)_{n \geq 1}$ of
$\cH$-valued random element converges
weakly to $h \in \cH$ as $n \to \infty$,
written as 
``$h_n \wconv h$ in $\cH$'', if $\lim_{n \to \infty} \IE  f( h_n )  = \IE  f (h) $ for every continuous and bounded function $f: \cH \to \mathbb{R}$. 

Although multigraphons have been introduced by \cite{Kol11}, the characterisation of weak convergence for multigraphon sequences has not been discussed in the literature to the best of our knowledge. The following theorem provides some equivalent conditions of the weak convergence of multigraphon sequences, which is a generalization of Theorem 3.1 in \citet{Dia07}.
\begin{theorem}
	\label{thm0}
	Let $h, h_1, h_2, \ldots \in (\cH, \dms)$ be a sequence of random multigraphons.
	Then the following are equivalent: 
	\begin{enumerate}[(\roman*)]
		\item $h_n \wconv h$ in $(\cH, \dms)$ as $n \to \infty$; 
		\item for every $F \in \cM$, we have $t_F(h_n) \wconv t_F(h)  \text{ in $\IR$}$ as $n \to \infty$;
		\item for every $F \in \cM$, we have $\lim_{n \to \infty}\IE \{t_F(h_n)\} = \IE \{t_F(h)\}$;
 	\end{enumerate}
\end{theorem}

\begin{proof}
\noindent{\textit{(i)}$\implies$\textit{(ii)}.}
By the definition of $\dms$, it follows that for any nonrandom $F\in \cM$, the map $t_F(\cdot) : (\cH, \dms) \to \IR$ is continuous. By the continuous mapping theorem, we have \textit{(i)} implies \textit{(ii)}.
	
\medskip\noindent{\textit{(ii)}$\implies$\textit{(iii)}.} 
This is a consequence of the bounded convergence theorem. 

\medskip\noindent{\textit{(iii)}$\implies$\textit{(i)}.}
For any $F_1, F_2 \in \mathcal{M}$ with $v(F_1) = k_1$ and $v(F_2) = k_2$, and denoting by $A_1 = (a_{1; i,j})_{1 \leq i, j \leq k_1}$ and $A_2 = (a_{2; i,j})_{1 \leq i, j \leq k_2}$ by their adjacency matrices, respectively. We have by definition that 
\begin{align*}
		t_{F_1}(h) t_{F_2}(h)
		& = \biggl( \int_{[0,1]^{k_1}} \prod_{1 \leq i \leq j \leq k_1} \sum_{r \geq a_{1; i,j}} h(r; x_i, x_j) dx_1 \dots d x_{k_1}	\biggr) \\
		& \quad \times \biggl( \int_{[0,1]^{k_2}} \prod_{1 \leq i' \leq j' \leq k_2} \sum_{r' \geq a_{2; i',j'}} h(r'; y_{i'}, y_{j'}) dy_1 \dots dy_{k_2}\biggr)\\
		& = t_{F_1 \uplus F_2}(h)	, 
\end{align*}
where $F_1 \uplus F_2$ is the disjoint union of $F_1$ and $F_2$. As $F_1 \uplus F_2 \in \mathcal{M}$, it follows that the class $\{ t_{F} : F \in \mathcal{M} \}$ forms an algebra. Noting that $(\cH, \dms)$ is a complete and separable metric space, and by \cref{lem4} below and
\cite[Theorem~4.5(b),~p.~113]{Eth05}, we have $\{t_F, F \in \cM\} \subset {\cC}_b(\cH)$ is convergence determining, where ${\cC}_b(\cH)$ is the class of bounded and continuous functions from $(\cH, \dms)$ to $\IR$. Moreover, by \textit{(iii)}, we have $\lim_{n \toinf}\IE \{
	t_F(h_n)
\} = \IE \{ t_F(h) \}$ for all $F \in \cM$.	By \cite[Eq.~(4.4),~p.~112]{Eth05}, we conclude that $h_n \wconv h$.
\end{proof}

\begin{remark} It is tempting to interpret subgraph densities as the ``moments'' of random graphons. It may then come somewhat as a surprise that the family of functions $(t_F)_{F\in \cM}$ is convergence determining even though the space $\cH$ is not compact. In analogy to real-valued random variables, moments are convergence determining for probability measures on compact subsets of $\IR$, but they are in general not convergence determining for measures on the whole real line. The reason is in essence that polynomials are bounded functions on compact sets and rich enough to be convergence determining, but they are unbounded when seen as functions on the whole real line, and so do not fall within the usual framework of weak convergence. This is in contrast to subgraph densities, which are always bounded functions, and so interpreting subgraph densities simply as the analogue of moments of random variables does not fully capture the role they play in the theory of graphons and multigraphons.  
\end{remark}

The following lemma, used in the proof of \cref{thm0}, ensures that the family $\{ t_F(\cdot), F \in \cM \}$ strongly separates points in $(\cH, \dms)$.

\begin{lemma}
\label{lem4}
The family of functions $\{t_{F}(\cdot), F \in \cM\}$ strongly separates points in $(\cH, \dms)$.
\end{lemma}
\begin{proof}
We need to show that 
for each $h \in \cH$ and each $\eps >0$ there exists $m \geq 1$ such that 
\begin{equ}
	\label{eq2-32}
	\inf_{h' :\dms (h, h') \geq \eps } \max_{ 1 \leq i \leq m} \bigl\lvert t_{{F_i^*}} (h) - t_{{F_i^*}} (h')\bigr\rvert > 0, 
\end{equ}
where $({F_i^*})_{i \geq 1}$ is the enumeration of all multigraphs that generates the distance $\dms$.

Now, fix $h \in \cH$ and $\eps >0$. 
Recall that $K_{2,r}$ is the graph on two vertices with $r$ edges connecting them, and let $L_r$ is the graph on one vertex with $r$ loops. 
Let 
\be{
	d_{\sub} (h, h')  =  \sum_{i \geq 1} 2^{-i} \lvert t_{{F_i^*}} (h) - t_{{F_i^*}} (h')\rvert .
}
By \cref{eq2-21}, it follows that $\dms(h, h') \to 0$ as $d_{\sub}(h, h') \to 0$. Then,  there exists $\delta \coloneqq \delta(\eps)$ such that 
$d_{\sub}(h, h') < \delta$ implies $\dms(h, h') < \eps$. Therefore, $\{ h': \dms(h, h') \geq \eps \} \subset \{h' : d_{\sub}(h, h') \geq \delta\}$. 
Then, to show \cref{eq2-32}, it suffices to prove that
there exists $m \geq 1$ such that 
\begin{equ}
	\inf_{h' : \dsub(h, h') \geq \delta} \max_{1 \leq i \leq m} \bigl\lvert t_{{F_i^*}}(h) - t_{{F_i^*}}(h') \bigr\rvert > 0.
    \label{eq2-33}
\end{equ}
To this end, letting $m \coloneqq m(\delta)$ be the smallest integer such that 
$\sum_{i > m}2^{-i} < \delta/2$, 
we claim that for all $h'$ satisfying $\dsub(h, h') \geq \delta$,
\begin{equ}
	\max_{1 \leq i \leq m} \bigl\lvert t_{{F_i^*}}(h) - t_{{F_i^*}}(h') \bigr\rvert \geq \frac{\delta}{2 m},  
    \label{eq2-34}
\end{equ}
which implies \cref{eq2-33} and hence \cref{eq2-32}.

We prove \cref{eq2-34} by contradiction. If \cref{eq2-34} does not hold, then  
\begin{align*}
	\dsub(h, h') \leq \sum_{i = 1}^{m} \bigl\lvert t_{{F_i^*}}(h) - t_{{F_i^*}}(h') \bigr\rvert + \sum_{i > m} 2^{-i} < \delta, 
\end{align*}
which contradicts $\dsub(h, h') \leq \delta$. 
\end{proof}

The definition of weak convergence extends naturally to multigraph sequences $G_1,G_2,\dots$ through their multigraphon representation $h^{G_1},h^{G_2},\dots$, and we simply write $G_n \wconv h$ if $h^{G_n} \wconv h$. 
As $v(G)\toinf$,  $t_F(G)$, $t_F^{\inj}(G)$ and $t_F^{\ind}(G)$ are equivalent. 
These equivalence relations, together with \cref{thm0}, yields the following corollary. 
\begin{corollary}
	\label{cor0}
	Let $G_1, G_2, \ldots \in \cM$ be a sequence of random multigraphs defined on a probability space $(\Omega, \cF, \IP)$ with $v(G_n) \to \infty$ $\IP$-a.s.\ ($n \toinf$), and let $h \in \cH$ be a random multigraphon. Then the following are equivalent: 
	\begin{enumerate}[(\roman*)]
		\item $G_n \wconv h$ in $(\cH, \dms)$ as $n \to \infty$; 
		\item for every $F \in \cM$, we have 
			$t_F^{\inj}(G_n) \wconv t_F(h) \text{ in $\IR$}$ as $n \to \infty$;
		\item for every $F \in \cM$, we have
				$t_F^{\ind}(G_n) \wconv t_F^{\ind}(h)$ in $\IR$ as $n \to \infty$; 
		\item for every $F \in \cM$, we have $\lim_{n \to \infty}\IE \bclc{t_F^{\inj}(G_n)} \to \IE \{t_F(h)\}$; 
		\item for every $F \in \cM$, we have $\lim_{n \to \infty}\IE \bclc{t_F^{\ind}(G_n) } \to \IE \bclc{t_F^{\ind}(h) }$. 
	\end{enumerate}
\end{corollary}

\subsection{Multigraphon-valued stochastic processes}%
\label{sub2.3}
Let $\cD \coloneqq \cD ( [0, \infty), \cH  )$, the càdlàg paths in $(\cH,\dms)$. Let $\kappa$ be a $\cH$-valued stochastic process. We write $\kappa(s)$ to denote the value of the process at time $s \geq
0$, which is an element of $\cH$. For any $\kappa \in \cD$ and $F \in \cM$, we denote by $t_{F } (\kappa)$ the induced stochastic process defined as $t_{F } (\kappa) (s) = t_F(\kappa(s))$. By definition, it follows that $t_{F } (\kappa)$ takes values in $\cD ( [0, \infty), [0,1] )$. 

We proceed to define the Skorohod topology on $\cD$ in the usual way. Let 
\bmn{
	\Lambda = \bigl\{ \lambda: [0, \infty) \to [0, \infty): \\
	\lambda \text{ is onto and increasing satisfying that $\gamma(\lambda) < \infty$} \bigr\}, 
}%
where
\be{
  \gamma(\lambda) \coloneqq \sup_{0<s_1 < s_2} \biggl\lvert \log \frac{\lambda(s_2) - \lambda(s_1)}{s_2 - s_1} \biggr\rvert.
}
We equip the space $\cD$ with the distance $\dcir$ defined as 
\begin{equation}
	\label{eq-dcir}
	\dcir ( \kappa_1, \kappa_2) = \inf_{\lambda \in \Lambda} \biggl\{ \gamma(\lambda) \vee \int_0^{\infty} e^{-u} (\sup_{s \geq 0} \dms( \kappa_1(s \wedge u), \kappa_2( \lambda(s) \wedge u ) ) \wedge 1) du \biggr\}.
\end{equation}
Again, we use ``$\wconv$'' to denote weak convergence with respect to the underlying (pseudo)metric space.

We have the following characterization of weak convergence in terms of subgraph densities. 

\begin{theorem}
\label{thm1}
Let $\kappa, \kappa_1, \kappa_2, \dots $ be random elements in $\cD$. 
Then the following are equivalent:
\begin{enumerate}[label={(\roman*)}]		
	\item  $\kappa_n \wconv \kappa$ in $(\cD, \dcir)$ as $n \to \infty$; 
	\item for every $q \geq 1$ and every  $F_1, \dots, F_q \in \cM$, we have 
	\be{
		\bigl(t_{F_1} (\kappa_n), \dots, t_{F_q}(\kappa_n)  \bigr) \wconv \bigl( t_{F_1} (\kappa), \dots, t_{F_q} (\kappa) \bigr) \text{ in $\cD( [0, \infty), \IR^q )$} 
	}%
	as $n \to \infty$;
	\item  for every $F \in \cM$, the sequence $(t_{F} (\kappa_n))_{n \geq 1}$ is tight, and for every $q \geq 1$, all real numbers $0 \leq s_1 < \dots < s_q < \infty$ where $\kappa$ is continuous almost surely, and every $F_1, \dots, F_q \in \cM$, we have 
		\begin{align*}
			\lim_{n \to \infty} \IE \{ t_{F_1} (\kappa_n (s_1)) \dots t_{F_q} (\kappa_n (s_q))   \} = \IE \{ t_{F_1} (\kappa (s_1)) \dots t_{F_q} (\kappa (s_q))   \}  .
		\end{align*}
\end{enumerate}
\end{theorem}

\begin{proof}
We apply several results from \cite{Eth05}, and use \cref{lem0,lem4}.

{\medskip\noindent\textit{(i)}$\implies$\textit{(ii)}.}%
\label{par2}
By the definition of $\dms$ in \cref{eq2-1,eq2-6}, it follows that the homomorphism map $t_F$ is continuous from $\cD$ to $\cD( [0, \infty], \IR )$. By the continuous mapping theorem (c.f.\ \cite[Problem 13, p.\,151]{Eth05}, we have \textit{(i)} implies \textit{(ii)}. 

{\medskip\noindent\textit{(ii)}$\implies$\textit{(iii)}.}%
\label{par3}
For $q = 1$, it follows from $(ii)$ that $t_{F} (\kappa_n) \wconv t_{F} (\kappa)$, which implies that $(t_F (\kappa_n))_{n \geq 1}$ is tight. By the definition of weak convergence, for the points of almost sure continuity of $\kappa$, the finite dimensional convergence in \textit{(iii)} follows from \textit{(ii)}. 

{\medskip\noindent\textit{(iii)}$\implies$\textit{(i)}.}%
\label{par4}
Let ${\cC}_b(\cH)$ be the family of bounded and continuous functions that maps from $\cH$ to $\mathbb{R}$ and let $\cF \coloneqq\{t_F : F \in \cM\}$; clearly, $\cF \subset \cC_b(\cH)$.
By \cref{lem4}, we have that the family $\cF$ strongly separates points in $(\cH, \dms)$. By the assumption of \textit{(iii)}, we have that $(t_F(\kappa_n))_{n \geq 1}$ is tight for every $F \in \cM$. 
Recall that $(\cH, \dms)$ is a complete and separable metric space. 
By \cite[p.\,153, Problem\,24]{Eth05}, we have $\kappa_n \wconv \kappa$ follows from the convergence of finite dimensional distribution of $\kappa_n$ to that of $\kappa$.

Now, it suffices to prove the convergence of finite-dimensional of $\kappa_n$. 
By \cref{lem4} and by \cite[Theorem 4.5(b), p.~113]{Eth05}, $\{ t_F, F \in \cM \}$ is convergence determining. By \cite[Proposition 4.6(b), p.~115]{Eth05}, functions of the form $t_{F_1}  \cdots   t_{F_q}$ are convergence determining on the product space \(
(\cH)^q\) with the metric $\dms$, and so convergence of finite dimensional distributions follows. This establishes \textit{(i)}.
\end{proof}

Let $(\bm G_n)_{n \geq 1} \subset \cD( [0, \infty), \cM )$ be a sequence of multigraph-valued processes; we denote by $G_n(s)$ the value of $\bm G_n$ at time $s$, which is a multigraph. 
We write $\bm G_n \wconv \kappa$ if the induced $\cH$-valued process $\kappa^{\bm G_n}$ converges weakly to $\kappa$. 
For any $\bm G = (G(s))_{s \geq 0} \in \cD( [0, \infty), \cM )$ and $F \in \cM$, we let $t_F(\bm G)$, $t_F^{\inj}(\bm G)$, $t_F^{\ind}(\bm G)$ be the induced stochastic processes with paths in $\cD( [0, \infty), [0,1] )$ defined as $t_F(\bm G)(s) = t_F( G(s)
)$, $t_F^{\inj}(\bm G)(s) = t_F^{\inj}(G(s))$ and $t_F^{\ind}(\bm G)(s) = t_F^{\ind}(G(s))$. 
The following corollary provides some additional equivalent conditions for the weak convergence in terms of functionals
$t_F^{\inj}$ and $t_{F}^{\ind}$, which are direct consequences of \cref{thm1}.

\begin{corollary}
	\label{cor1}
Let $\bm G_1, \bm G_2, \ldots \in \cD( [0, \infty), \cM )$ be a sequence of multigraph-valued stochastic process such that 
\be{
	\inf_{s \geq 0} v(G_{n} (s)) \to \infty \quad (n \to \infty),
}%
where $v(G)$ is the number of vertices of $G$.
Let $\kappa$ be a random element in $\cD$. Then the following are equivalent:
\begin{enumerate}[$(i)$]
	\item $\bm G_n \wconv \kappa$ in $(\cD, \dcir)$ as $n \to \infty$; 
	\item $\bigl(t_{F_1}^{\inj} (\bm G_n), \dots, t_{F_q}^{\inj}(\bm G_n)  \bigr) \wconv \bigl( t_{F_1} (\kappa), \dots, t_{F_q} (\kappa) \bigr)$ in $\cD([0, \infty), \IR^q)$ as $n \to \infty$ for all $q \geq 1$ and all multigraphs $F_1, \dots, F_q \in \cM$;
	\item $\bigl(t_{F_1}^{\ind} (\bm G_n), \dots, t_{F_q}^{\ind}(\bm G_n)  \bigr) \wconv \bigl( t_{F_1}^{\ind} (\kappa), \dots, t_{F_q}^{\ind} (\kappa) \bigr)$ in $\cD([0, \infty), \IR^q)$ as $n \to \infty$ for all $q \geq 1$ and every  $F_1, \dots, F_q \in \cM$;
	\item for every $F \in \cM$, the sequence $\bclr{t_{F}^{\inj} (\bm G_n)}_{n \geq 1}$ is tight, and for every $q \geq 1$, all real numbers $0 \leq s_1 < \dots < s_q < \infty$ where $\kappa$ is continuous almost surely, and every  $F_1, \dots, F_q \in \cM$, we have 
  \be{
			\lim_{n \to \infty} \IE \{ t_{F_1}^{\inj} (G_n (s_1)) \dots t_{F_q}^{\inj} (G_n (s_q))   \} = \IE \{ t_{F_1} (\kappa (s_1)) \dots t_{F_q} (\kappa (s_k))   \}  .
	}%
	\item for every $F \in \cM$, the sequence $(t_{F}^{\ind} (\bm G_n))_{n \geq 1}$ is tight, and for all $q \geq 1$, all real numbers $0 \leq s_1 < \dots < s_q < \infty$ where $\kappa$ is continuous almost surely, and every $F_1, \dots, F_q \in \cM$, we have 
  \be{
			\lim_{n \to \infty} \IE \{ t_{F_1}^{\ind} (G_n (s_1)) \dots t_{F_q}^{\ind} (G_n (s_q))   \} = \IE \{ t_{F_1}^{\ind} (\kappa (s_1)) \dots t_{F_q}^{\ind} (\kappa (s_k))   \}  .
	}%
\end{enumerate}
\end{corollary}

\subsection{Erased graphs generated from multigraphs}
\label{sub3}
In this subsection, we consider graphs that are simple graphs obtained from multigraphs by removing loops and merging multiple edges; we call these graphs \emph{erased graphs} (see, e.g., \cite[Chapter~7]{van17b}). Specifically, let $G = (z_{ij})_{ij \in [n]} \in \cM$ be a multigraph. The corresponding erased graph $\hat{G} = (\hat{z}_{ij} )_{i, j \in [n]}$ of $G$ is defined as
\begin{align*}
	\hat{z}_{ij} = 
	\begin{cases}
		\1 \{ z_{ij} \geq 1 \}, & i \neq j, \\
		0, & i = j. 
	\end{cases}
\end{align*}
The weak limiting behavior of simple graphon-valued stochastic process has been studied by \citet*{Ath19}. We now introduce some notation.
Let $\cW$ be the space of graphons. We say $h_1 , h_2 \in \cW$ are equivalent if there exists two measure-preserving bijections $\sigma_1$ and $\sigma_2$ such that $h_1(\sigma_1x, \sigma_1y) = h_2(\sigma_2 x, \sigma_2 y)$. This equivalence relation yields the quotient space $\Wtilde$. Let $\cD(
[0, \infty), \Wtilde )$ be the set of c\`adl\`ag paths in  $\Wtilde $. 

Let $h$ be a multigraphon; we define its \emph{erased graphon} $\hhat : [0, 1]^2 \to [0, 1]$ by 
\begin{align*}
	\hhat(x, y) = \sum_{r = 1}^{\infty} h(r; x, y). 
\end{align*}
Similarly, for $\kappa \in \cD$, we define the $\Wtilde$-valued process $\hat{\kappa}$ as at each $s \geq 0$, the element $\hat{\kappa}(s) \in \Wtilde$ is the equivalence class of the erased graphon of $\kappa(s)$.

\begin{corollary}
\label{thm2} Let $\kappa, \kappa_1, \kappa_2, \dots $ be a sequence of stochastic processes in $\cD([0,\infty,\cH)$, and let $\hat{\kappa}, \hat{\kappa}_1, \hat{\kappa}_2, \ldots$ be the corresponding erased processes in  $\cD( [0, \infty), \Wtilde  )$. 
If $\kappa_n \wconv \kappa$ in $\cD([0,\infty,\cH)$, then $\hat{\kappa}_n \wconv \hat{\kappa}$ in $\cD( [0, \infty), \Wtilde )$.
\end{corollary}

\begin{proof}
[Proof of \cref{thm2}] 
Let $t^{\simple}$ be the homomorphism density for simple graphons; that is, for any $h \in \cW$, and any simple graph $F$ with $k$ vertices, let 
\begin{align*}
	t_F^{\simple}(h) = \int_{[0,1]^k} \prod_{ij \in F} h(x_i, x_j) dx_1 \dots d x_k.  
\end{align*}
By Theorem 3.1 of \citet*{Ath19}, it suffices to prove the following two conditions: 

{\medskip\noindent\it (i) Tightness.} For every graph $F \in \cF$, the sequence $(t_F^{\simple} (\hat{\kappa}_n))_{n\geq 1}$ is tight. 

{\medskip\noindent\it (ii) Finite dimensional convergence.} For all $q \geq 1$, all $0 \leq s_1 < \dots < s_q < \infty $ and all $F_1, \dots, F_q \in \cF$, 
		\begin{align*}
			\lim_{n \to \infty} \IE \{  t^{\simple}_{F_1} ( \hat{\kappa}_n(s_1) ) \cdots t^{\simple}_{F_q} ( \hat{\kappa}_n (s_q) )  \} = \IE \{    t^{\simple}_{F_1} ( \hat{\kappa}(s_1) ) \cdots t^{\simple}_{F_q} ( \hat{\kappa} (s_q) )  \}.
		\end{align*}
Now, for $\kappa$ and $(\kappa_j)_{j \geq 1}$, consider the truncated multigraphon processes $\bar{\kappa} $ and $(\bar{\kappa}_j)_{j \geq 1} $, that are defined by, for $x \neq y$, 
\begin{gather*}
	\bar{\kappa} (s; 0; x, y) = \kappa(s; 0; x, y), \quad \bar{\kappa} (s; 1; x, y) = \sum_{r = 1}^{\infty} \kappa(s; r; x, y), \\
	\bar{\kappa}_j (s; 0; x, y) = \kappa_j(s; 0; x, y), \quad \bar{\kappa}_j (s; 1; x, y) = \sum_{r = 1}^{\infty} \kappa_j (s; r; x, y),
\intertext{and}
\bar{\kappa} (s; 0; x, x) = 1, \quad \bar{\kappa} (s; 1; x, x) = 1 , \quad 
\bar{\kappa}_j (s; 0; x, x) = 1, \quad \bar{\kappa}_j (s; 1; x, x) = 0 .
\end{gather*}
Then, it is easy to check that the map $\kappa \in \cD \mapsto \bar{\kappa} \in \cD$ is continuous. 

By the construction of $\hat{\kappa} $ and $(\hat{\kappa}_j)_{ j \geq 1}$, we have for any simple graph $F$,
\begin{align*}
	t_F^{\hspace{0.1ex}\simple} (\hat{\kappa}) = t_F (\bar{\kappa}), \quad t_F^{\hspace{.1ex}\simple} (\hat{\kappa}_j) = t_F (\bar{\kappa}_j).
\end{align*}
Since $\kappa_n \wconv \kappa$ as $n\toinf$, and by the continuous mapping theorem, we have
$\bar{\kappa}_n \wconv \bar{\kappa} $ as $n\toinf$. 
Thus, we have \emph{(i)} and \emph{(ii)} are satisfied, and hence the theorem is proved. 
\end{proof}

\section{Dynamics on configuration random multigraphs}%
\label{sec1}

\subsection{Configuration model}%
\label{sub6}

The configuration model was originally introduced by \cite{Ben78} and \cite{bollobas1980}, who considered a uniform simple $d$-regular graph on $n$ nodes. This model was later generalized by \cite{Mol95}, who obtained conditions for the existence of a giant component; we refer to \cite{van17b} for an in-depth discussion.

We proceed with the mathematical definition of the model. Let $n \geq 1$ be an integer and let $d_n = (d_{n, 1}, \dots, d_{n, n} )$ be a sequence of positive integers. 
Let $\ell_n \coloneqq \sum_{i = 1}^n d_{n, i}$ be the
sum of all degrees; we assume that $\ell_n$ is even. To construct a multigraph where vertex $j$ has degree $d_{n,j}$ we start with $n$ vertices, where vertex $j$ has $d_{n,j}$ half-edges for $1 \leq j \leq n$. We further assume that the half-edges are numbered in an arbitrary order from $1$ to $\ell_n$. We construct the configuration random multigraph as follows.  
Connect the first half-edge with one of the $\ell_n - 1$ remaining ones, chosen uniformly at random. Continue the procedure for the remaining half-edges until all of them are connected. The distribution of the resulting multigraph on the set $\cM_n$ is denoted by $\CM(d_n)$.

Let $D_n \coloneqq (D_{n,1}, \dots, D_{n,n})$ be a random degree sequence defined on a probability space $(\Omega, \mathcal{F}, \IP)$. We proceed to prove that $G_n \sim \CM(D_n)$
converges in distribution to a random multigraphon, and the limiting multigraphon depends on the limiting behaviour of the random degree sequence $D_n$. To specify the limiting multigraphon, we need to introduce some assumptions.
Let $L_n = \sum_{i = 1}^n D_{n,i}$ and $Y_n = L_n / n^2$. For $k \geq 1$, let $Z_{n,1}, \dots, Z_{n,k}$ be a simple random sample from the set $\{n D_{n, 1} / L_n , \dots, n D_{n, n} / L_n  \}$, chosen uniformly and without replacement. 
Assume that for each $k \geq 1$, there exists a vector of random variables $(Z_1, \dots, Z_k, Y)$ such that, as $n \toinf$, 
\begin{equation}
	(Z_{n,1}, \dots, Z_{n,k}, Y_n) \wconv (Z_1, \dots, Z_k, Y) \quad  \text{in $\IR^k \times \IR$},
    \label{eq-lim}
\end{equation}
where $Z_1, \dots, Z_k$ are conditionally independent given $Y$, and have a common distribution function $\Psi$. Here, $\Psi$ may depend on $Y$.
Define the generalised inverse of $\Psi$ as 
\begin{equation}
	\bar \Psi(x) = \inf \{ y: \Psi(y) \geq x \}, \quad x \in [0,1].
    \label{eq-barpsi}
\end{equation}

Now, we are ready to define the limiting multigraphon. 
Let 
\begin{equation}
	\label{eq19}
	h(r; x, y) = 
	\begin{dcases}
		{p} \bigl( r ; Y \bar\Psi(x) \bar\Psi(y)\bigr)   & \text{if $x \neq y$} , \\
		{p} \Bigl( \frac{r}{2} ; \frac{ Y\bar\Psi (x)^2 }{2} \Bigr)         & \text{if $x = y$ and if $r$ is even}   , \\
		0                                      & \text{otherwise}  ,
	\end{dcases}
\end{equation}
where ${p} (r; \lambda) = e^{-\lambda } \lambda^k / k!$ for $r \geq 0$ 
and where $\-\Psi$ and $Y$ are as in \cref{eq-lim,eq-barpsi}.
\begin{theorem}
	\label{t3.1}
	Let $G_n  \sim \CM(D_n)$, and let $L_n \coloneqq  \sum_{i = 1}^n D_{n,i}$. 
Assume that \cref{eq-lim} holds and that 
\begin{equation}
	L_n \geq n \quad \text{and} \quad \max_{1 \leq i \leq n} D_{n,i} / (L_n^{1/2} (\log n)^{2}) \to 0 \quad \IP\text{-a.s.}
    \label{eq-rho}
\end{equation}
Then  $G_n \wconv h$. 
\end{theorem}

Before proving \cref{t3.1}, we first prove a lemma. 
\begin{lemma}
\label{lem:prob}For each $n\geq 1$, let $d_n$ be a degree sequence, and let $\ell_n \coloneqq \sum_{i = 1}^n d_{n,i}$. Assume that $\ell_n\geq n$ and that 
\begin{equation}
	\max_{1 \leq i \leq n }d_{n,i} / ( \ell_n^{1/2}(\log n)^{2})  \to 0 \quad  \text{as $n \to \infty$}.
	\label{eq-conlp}
\end{equation}
Let $G_n = (G_{n,ij})_{i,j\in [n]} \sim \CM(d_n)$. 
Then, for any $\sigma_n$ and any multigraph $F = (a_{ij})_{1 \leq i \leq j \leq k} \in \mathcal{M}_k$, we have
\begin{multline}
	\label{100}
		 \biggl\lvert \IP [ G_{n,\sigma_n} = F ]- \prod_{1 \leq i < j \leq k} p( a_{ij}; y_n z_{n, \sigma_n(i)} z_{n, \sigma_n(j)} ) \prod_{i = 1}^k p \Bigl( \frac{a_{ii}}{2}; \frac{y_n z_{n,\sigma_n(i)}^2}{2} \Bigr) \biggr\rvert \\
		 \leq C_F n^{-1/4},
\end{multline}
where $\{G_{n, \sigma_n} = F\}$  is the event that $G_{n,\sigma_n(i), \sigma_n(j)} = a_{ij}$ for all $1 \leq i \leq j \leq k$,  
and where
	\begin{math}
		 z_{n,j} = n d_{n,j}/\ell_n
	\end{math} for $1 \leq j \leq n$, $y_n = \ell_n / n^2$ and $C_F > 0$ is a constant depending only on $F$. 
\end{lemma}
\begin{proof}
Let $d_i(F)$ be the degree of node $i$ in $F$ and let
$\ell(F) = \sum_{i = 1}^k d_{i}(F)$.
Let 
\begin{math}
	c(F) = ({\prod_{1 \leq i < j \leq k}a_{ij}! \prod_{i = 1}^k a_{ii}!!})^{-1}.
\end{math}
Rewriting the second term of the left hand side of \cref{100} gives 
\begin{multline*}
		\prod_{1 \leq i < j \leq k} p( a_{ij}; y_n z_{n, \sigma_n(i)} z_{n, \sigma_n(j)} ) \prod_{i = 1}^k p \Bigl( \frac{a_{ii}}{2}; \frac{y_n z_{n,\sigma_n(i)}^2}{2} \Bigr)  \\
		 =  c(F) \exp \biggl( - \frac{1}{2}  \Bigl(\sum_{j = 1}^k \frac{d_{n,\sigma_n(j)}}{\ell_n^{1/2}}\Bigr)^2 \biggr)  \prod_{i = 1}^k \ell_n^{-d_i(F)/2} d_{n, \sigma_n(i)}^{d_i(F)}.
\end{multline*}
Thus, it suffices to prove that for large $n$, 
\begin{multline}
	 \biggl\lvert \IP [G_{n, \sigma_n} = F] 
	 -  c(F) \exp \biggl( - \frac{1}{2}  \Bigl(\sum_{j = 1}^k \frac{d_{n,\sigma_n(j)}}{\ell_n^{1/2}}\Bigr)^2 \biggr)  \prod_{i = 1}^k \ell_n^{-d_i(F)/2} d_{n, \sigma_n(i)}^{d_i(F)} \biggr\rvert \\
	 \leq C_Fn^{-1/4}.
	\label{p1-05}
\end{multline}
Let 
\begin{math}
		\ell_{\sigma_n} = \sum_{i = 1}^k d_{n, \sigma_n(i)}. 
\end{math}
Now, by \citet[Eqs.~(49)~and~(50)]{Rat12}, we have 
\besn{	
  &\IP[G_{n, \sigma_n} = F] \\
	&\quad=c(F)\frac{ \prod_{i = 1}^k d_{n,\sigma_n(i)}! }{\prod_{i = 1}^k (d_{n, \sigma_n(i)} - d_i(F))!  } 
	\frac{(\ell_n /2 )! 2^{\ell_{\sigma_n} - \ell(F)/2}  }{ ( \ell_n/2 - \ell_{\sigma_n} + \ell(F)/2)!} \frac{ (\ell_n - \ell_{\sigma_n})!
	}{\ell_n!}.
      \label{p1-00}
}
The rest of the proof includes two steps. 

\medskip 
{\noindent \bf Step 1.} We show that there exists $n_1 > 1$ such that for all $n \geq n_1$, 
\begin{equ}
	\biggl\lvert \frac{(\ell_n/2)! 2^{\ell_{\sigma_n} - \ell(F)/2}}{(\ell_n/2 - \ell_{\sigma_n} + \ell(F)/2)!} \frac{ (\ell_n - \ell_{\sigma_n})! }{\ell_n!} - (\ell_n)^{-\ell(F)/2}e^{-\ell_{\sigma_n}^2 / (2 \ell_n)} \biggr\rvert \leq C_F n^{-1/3}.
      \label{eq-Step1}
\end{equ}
To this end, 
we use the well-known Stirling's approximation to estimate the first term of the left hand side of \cref{eq-Step1}: 
\begin{equation}
	\label{11l-0}
	\sqrt{2\pi} x^{x + 1/2} e^{-x + 1/(12 x + 1)}\leq \Gamma(x + 1) \leq \sqrt{2\pi} x^{x + 1/2} e^{-x + 1/(12x)}
\end{equation}	
for $x > 0$, where $\Gamma$ is the Gamma function. 
Rewriting the first term of \cref{eq-Step1} as 
\begin{align*}
	\frac{(\ell_n/2)! 2^{\ell_{\sigma_n} - \ell(F)/2}}{(\ell_n/2 - \ell_{\sigma_n} + \ell(F)/2)!} \frac{ (\ell_n - \ell_{\sigma_n})! }{\ell_n!} = I_1 \times I_2, 
\end{align*}
where 
\begin{align*}
	I_1 & =  \frac{2^{-\ell(F)/2}(\ell_n/2 - \ell_{\sigma_n})!}{(\ell_n/2 - \ell_{\sigma_n} + \ell(F)/2)!}, & 
	I_2 & = \frac{(\ell_n/2)! 2^{\ell_{\sigma_n}}}{(\ell_n/2 - \ell_{\sigma_n})!} \frac{(\ell_n - \ell_{\sigma_n})!}{\ell_n!}.
\end{align*}
By \cref{eq-conlp},  we have 
\begin{equation}
	\ell_{\sigma_n} / (\ell_n^{1/2} (\log n)^2) \leq k \max_{1 \leq i \leq n}d_{n,i}/(\ell_n^{1/2} (\log n)^2) \to 0 \quad \text{as $n \to \infty$}. 
	\label{eq-conlp2}
\end{equation}
Recalling the assumption that $\ell_n \geq n$, we have there exists $n_1 \geq 1$ such that for all $n \geq n_1$,
\begin{equ}
	\ell_{\sigma_n}/\ell_n \leq n^{-1/3} \leq 0.1, \quad \ell_n - 2 \ell_{\sigma_n} \geq \frac{1}{2} \ell_n. 
      \label{eq-11l-ell}
\end{equ}
By $\ell_n \geq n$ again, and by \cref{11l-0},  we have $\ell_n / 2 - \ell_{\sigma_n} \geq n /4$ for $n \geq n_1$, and
\begin{equ}
	I_1
	& \leq 
	\frac{(2e)^{-\ell(F)/2 } (\ell_n / 2 - \ell_{\sigma_n})^{\ell_n / 2 - \ell_{\sigma_n} + 1/2} }{(\ell_n / 2 - \ell_{\sigma_n} + \ell(F)/ 2)^{ \ell_{n}/2 - \ell_{\sigma_n} + \ell(F)/2 + 1/2 }} e^{1/(3n)} \\
	& = \frac{e^{-\ell(F)/2 } (\ell_n / 2 - \ell_{\sigma_n})^{\ell_n / 2 - \ell_{\sigma_n} + 1/2} }{(\ell_n / 2 - \ell_{\sigma_n} + \ell(F)/ 2)^{ \ell_{n}/2 - \ell_{\sigma_n} + \ell(F)/2 + 1/2 }}  (\ell_n - 2 \ell_{\sigma_n})^{-\ell(F)} e^{1/(3n)}.
    \label{eq-11l-I1}
\end{equ}
Moreover, we have for $n \geq n_1$, the fraction term on the right hand side of \cref{eq-11l-I1} can be bounded by 
\begin{align*}
	\biggl\lvert \frac{e^{-\ell(F)/2} (\ell_n/2 - \ell_{\sigma_n})^{\ell_n/2 - \ell_{\sigma_n} + \ell(F)/2 + 1/2}}{(\ell_{n}/2 - \ell_{\sigma_n} + \ell(F)/2)^{\ell_n/2 - \ell_{\sigma_n} + \ell(F)/2 + 1/2}} -  1 \biggr\rvert  \leq C_F n^{-1}.
\end{align*}
Therefore, 
\begin{equation}
	\begin{split}
		I_1 & \leq  \bigl({\ell_n} - 2\ell_{\sigma_n}\bigr)^{-\ell(F)/2} (1 + Q_1)  \\
		& \leq  \ell_n^{-\ell(F)/2}(1 + Q_2) ,
	\end{split}
      \label{p1-02}
\end{equation}
for some $\lvert Q_1 \rvert \leq C_F n^{-1}$ and $\lvert Q_2  \rvert \leq C_F n^{-1/3}$, and we used \cref{eq-11l-ell} in the last line.
Using a similar argument we obtain
\begin{equ}
	I_1 \geq \ell_n^{-\ell(F)/2} ( 1 - Q_2 ). 
    \label{p1-02'}
\end{equ}

Now we consider $I_2$. 
Observe that for $n \geq n_1$, 
\begin{equ}
	\biggl\lvert \frac{ \ell_n  }{ \ell_n -  2 \ell_{\sigma_n} } \frac{\ell_n - \ell_{\sigma_n}}{\ell_n} - 1 \biggr\rvert \leq \frac{2 \ell_{\sigma_n}}{ \ell_n  } \leq 2 n^{-1/3}.
      \label{eq-11l-39}
\end{equ}
By \cref{11l-0,eq-11l-39}, and noting that $\ell_n \geq n$, 
\begin{equation*}
	\begin{split}
		I_2
		& \leq \sqrt{\frac{ \ell_n }{\ell_n - 2 \ell_{\sigma_n}} \frac{\ell_n - \ell_{\sigma_n}}{\ell_n}}  \biggl( \frac{ (\ell_n - \ell_{\sigma_n})^{2 - 2\ell_{\sigma_n}/\ell_n} }{\ell_n (\ell_n - 2 \ell_{\sigma_n})^{1 - 2 \ell_{\sigma_n}/\ell_n}} \biggr)^{\ell_n/2}
		\exp(C_F n^{-1/3})\\
		& = \biggl( \frac{ (\ell_n - \ell_{\sigma_n})^{2 - 2\ell_{\sigma_n}/\ell_n} }{\ell_n (\ell_n - 2 \ell_{\sigma_n})^{1 - 2 \ell_{\sigma_n}/\ell_n}} \biggr)^{\ell_n/2}
		\exp({C_F n^{-1/3}})\\
		& = \biggl( \frac{ (1 - x_n)^{2 - 2x_n} }{(1 - 2x_n)^{1 - 2x_n}} \biggr)^{\ell_n/2} \exp({C_F n^{-1/3}}), 
	\end{split}
\end{equation*}
where  $x_n = \ell_{\sigma_n}/\ell_n$.
Let $\psi (x) = (1 - x)^{2 - 2x}/(1 - 2x)^{1 - 2x}$; by Taylor's expansion and recalling \cref{eq-11l-ell}, we have for $n \geq n_1$,
\begin{align*}
	\psi(x_n) = 1 - x_n^2 + \frac{\psi'''(\xi_n)}{6} x_n^3, 
\end{align*}
for some $\lvert \xi_n \rvert \leq 0.1$. A direct calculation implies $\sup_{\lvert x \rvert \leq 0.1} |\psi'''(x)| \leq 8$, and we have for $n \geq n_1$, 
\begin{align*}
	\psi(x_n) = 1 - x_n^2 + u_n x_n^3, 
\end{align*}
for some $|u_n| \leq 1.4$.
Moreover, recalling \cref{eq-11l-ell}, we have for $n \geq n_1$, 
\begin{align*}
	|(1 - x_n^2 + u_n x_n^3)^{\ell_n/2} 
	- e^{-\ell_{\sigma_n}^2 / (2 \ell_n)}| \leq C_Fn^{-1/3},
\end{align*}
and therefore, for $n \geq n_1$, 
\begin{align}
 I_2 - e^{-\ell_{\sigma_n}^2 / (2 \ell_n)} \leq C_F n^{-1/3}.
	\label{p1-03}
\end{align}
Similarly, for $n \geq n_1,$
\begin{align}
 I_2 - e^{-\ell_{\sigma_n}^2 / (2 \ell_n)} \geq - C_F n^{-1/3}.
	\label{p1-03'}	
\end{align}
By \cref{p1-02,p1-02',p1-03,p1-03'}, we obtain \cref{eq-Step1}. 

\medskip
{\noindent \bf Step 2.}
Let $b_n = \min\{ d_{n, \sigma_n(i)} : d_i(F) > 0 \}$. 
In this step, we prove \cref{p1-05} for the two cases that $b_n < \ell_n^{1/4}$ and $b_n \geq \ell_n^{1/4}$ separately. 
Observe that 
\begin{equ}
	\frac{\prod_{i = 1}^k d_{n, \sigma_n(i)}!}{\prod_{i = 1}^k ( d_{n, \sigma_n(i)} - d_i(F) )!} \leq \prod_{i = 1}^k d_{n, \sigma_n(i)}^{d_i(F)}. 
      \label{p1-04}
\end{equ}
By \cref{eq-conlp}, we have there exists $n_2 \geq 1$ such that for all $n \geq n_2$, 
\begin{equ}
	\max_{1 \leq i \leq n} d_{n,i} / \ell_n \leq 1. 
      \label{eq-l11-d1}
\end{equ}
By \cref{p1-00,eq-Step1,p1-04,eq-l11-d1}, we have for $n \geq \max \{ n_1, n_2 \}$, 
\begin{equation}
	\begin{split}
		\MoveEqLeft \Bigl\lvert \IP [ G_{n,\sigma_n} = F ]
		- c(F) \ell_n^{-\ell(F)/2} \exp \biggl( - \frac{\ell_{\sigma_n}^2}{2 \ell_n} \biggr) \prod_{i = 1}^k d_{n, \sigma_n(i)}^{d_i(F)} \Bigr\rvert\\
		& \leq  C_F  n^{-1/3} \ell_n^{-\ell(F)/2} \prod_{i = 1}^k d_{n, \sigma_n(i)}^{d_i(F)} \\
		& \leq C_Fn^{-1/3}. 
	\end{split}
	\label{p1-06}
\end{equation}
Noting that if $b_n < \ell_n^{1/4}$, then there exists $j^* \in \{1, \dots, k\}$ and $n_3 \geq 1$ such that  $d_{j^*}(F) \geq 1$ and  $d_{n, \sigma_n(j^*)} / \sqrt{\ell_n}  \leq \ell_n^{-1/4} \leq n^{-1/4}$ for $n \geq n_3$. Then it follows that for $n \geq n_3$, 
\begin{equation}
	\begin{split}
		\MoveEqLeft \biggl\lvert \ell_n^{-\ell(F)/2} \exp \biggl( - \frac{\ell_{\sigma_n}^2}{2 \ell_n} \biggr) \prod_{i = 1}^k d_{n, \sigma_n(i)}^{d_i(F)} \biggr\rvert\\
		& = \biggl\lvert \exp \biggl( - \frac{1}{2}  \Bigl(\sum_{i = 1}^k \frac{d_{n,\sigma_n(i)}}{\ell_n^{1/2}}\Bigr)^2 \biggr)  \prod_{i = 1}^k \ell_n^{-d_i(F)/2} d_{n, \sigma_n(i)}^{d_i(F)}  \biggr\rvert\\
		& \leq C_F \exp \biggl( - \frac{d_{n,\sigma_n(j^*)}}{2 \ell_n^{1/2}}  \biggr)	\ell_{n}^{-d_{j^*}(F)/2}d_{n, \sigma_n(j^*)}^{d_{j^*}(F)} 
		 \leq C_F n^{-1/4}. 
	\end{split}
	\label{p1-07}
\end{equation}
Therefore, if $b_n < \ell_n^{1/4}$, by \cref{p1-06,p1-07}, we have for $n \geq \max \{ n_1, n_2, n_3 \}$, 
\begin{equation}
	\IP [G_{n,\sigma_n} = F] \leq C_F n^{-1/4}. 
	\label{p1-08}
\end{equation}
Combining \cref{p1-07,p1-08} we have that \cref{p1-05} holds for $b_n < \ell_n^{1/4}$ and $n \geq \max \{ n_1, n_2, n_3\}$. 

If $b_n \geq \ell_n^{1/4}$, then it follows that $b_n \geq n^{1/4}$. By Stirling's formula \cref{11l-0}, we have 
\begin{equ}
	\frac{\prod_{i = 1}^k d_{n, \sigma_n(i)}!}{\prod_{i = 1}^k ( d_{n, \sigma_n(i)} - d_i(F) )!} & =(1 + Q_3) \prod_{i = 1}^k d_{n, \sigma_n(i)}^{d_i(F)} ,
      \label{p1-01}
\end{equ}
for some $\lvert Q_3 \rvert \leq C_F n^{-1/4}$.
Also, by \cref{eq-conlp2}, we have  there exists $n_4 \geq 1$ such that $d_{n, \sigma_n(i)} \leq \ell_n^{2/3}$ for all $n \geq n_4$. Thus, for $n \geq n_4$, 
\begin{align*}
	\prod_{i = 1}^k \Bigl( \ell_n^{-d_i(F)} d_{n,\sigma_n(i)}^{d_i(F)}\Bigr) \leq C_F  n^{-1/3}.
\end{align*}
Substituting \cref{eq-Step1,p1-03} to \cref{p1-00}, we have for $n \geq \max \{ n_1, n_2, n_4 \}$, 
\begin{align*}
\begin{split}
	\MoveEqLeft 
	\biggl\lvert \IP [G_{n, \sigma_n} = F ] - c(F) \exp \biggl( - \frac{1}{2}  \Bigl(\sum_{j = 1}^k \frac{d_{n,\sigma_n(j)}}{\ell_n^{1/2}}\Bigr)^2 \biggr)  \prod_{i = 1}^k \ell_n^{-d_i(F)/2} d_{n, \sigma_n(i)}^{d_i(F)} \biggr\rvert\\
	& \leq  C_F n^{-1/4}\exp \biggl( - \frac{1}{2}  \Bigl(\sum_{j = 1}^k \frac{d_{n,\sigma_n(j)}}{\ell_n^{1/2}}\Bigr)^2 \biggr)  \prod_{i = 1}^k \ell_n^{-d_i(F)/2} d_{n, \sigma_n(i)}^{d_i(F)}  + C_F n^{-1/4} \\
	& \leq C_F n^{-1/4}.
\end{split}
\end{align*}
Then, \cref{p1-05} also holds if $b_n \geq \ell_n^{1/4}$ and $n \geq \max \{ n_1, n_2, n_4 \}$. This proves \cref{p1-05} for $n \geq \max \{ n_1, n_2, n_3, n_4 \}$. 
\end{proof}

\begin{proof}
	[Proof of \cref{t3.1}]
Denote
by $\IE_{D_n}$ and $\IP_{D_n}$, respectively, the conditional expectation operator the conditional probability operator, respectively, given $D_n$. 
	By \cref{eq2-13,lem:prob}, we have for any $k \geq 1$ and $F = (a_{ij})_{1 \leq i \leq j \leq k} \in \mathcal{M}_k$,
	\begin{multline}
		\label{eq-t31-1}
			\biggl\lvert \IE_{D_n} \{ t_{F}^{\ind}(G_n) \}- \IE_{D_n} \biggl\{ \prod_{1 \leq i < j \leq k} p( a_{ij}; Y_n Z_{n,i}Z_{n,j} ) \prod_{i = 1}^k p \Bigl( \frac{a_{ii}}{2}; \frac{Y_n Z_{n,i}^2}{2} \Bigr) \biggr\} \biggr\rvert\\
			\leq  C_F n^{-1/4},
	\end{multline}
	where $Z_{n,1}, \dots, Z_{n,k}$ are independently chosen with replacement from the set $\{n D_{n,1}/L_n, \dots, n D_{n,n}/L_n  \}$ and $Y_n = L_n / n^2$. 
	By \cref{eq-lim},  
	\begin{align}
		\begin{split}
			\MoveEqLeft \lim_{n \to \infty} \IE \biggl\{ \prod_{1 \leq i < j \leq k} p( a_{ij}; Y_n Z_{n,i}Z_{n,j} ) \prod_{i = 1}^k p \Bigl( \frac{a_{ii}}{2}; \frac{Y_n Z_{n,i}^2}{2} \Bigr) \biggr\}\\
			& = \IE \biggl\{ \prod_{1 \leq i < j \leq k} p( a_{ij}; Y \bar{\Psi}(U_i) \bar{\Psi}(U_j) ) \prod_{i = 1}^k p \Bigl( \frac{a_{ii}}{2}; \frac{Y \bar{\Psi}(U_i)^2}{2} \Bigr) \biggr\},
            \label{eq-t31-2}
		\end{split}
	\end{align}
	where  $U_1, \dots, U_k$ are independent random variables uniformly distributed on $[0,1]$ and also independent of all others. By \cref{eq2-4,eq-t31-1,eq-t31-2} we obtain 
\begin{equation}
	\label{eq18}
	\lim_{n \to \infty} \IE \{ t_F^{\ind} (G_n) \} = 	
	\IE \{t_{F}^{\ind} (h) \} \quad \text{for all $F \in \mathcal{M}$.}
\end{equation}
	By \emph{(iii)} of \cref{cor0}, we conclude that $G_n \wconv h$, which completes the proof. 
\end{proof}

\subsection{Edge reconnection model: A dynamic network model}%
\label{sec2}
In this subsection, we consider a dynamic network model, which we call the \emph{edge reconnection model}. This dynamic model is based on a random multigraph growth process,
which was introduced by \cite{Pit10} and further studied by \cite{Bor11b} and \cite{Rat12}.

The random multigraph growth model is defined as follows. 
Let $n \geq 1$ and let $\theta > 0$. 
Let $H_n(0)$ be the empty graph on the vertex set $[n]$. For $m \geq 0$, and given $H_n(m)$ having the degree sequence $d_n = ( d_{n,1} , \dots, d_{n,n})$, we construct $H_n(m + 1)$ by adding a new edge $(i, j)$ with  the following
preferential-attachment-type probability:  
\begin{align}
	\begin{dcases}
		\frac{ 2 (d_{n,i} + \theta ) (d_{n,j} + \theta ) }{ ( 2m + n \theta) (2m + n \theta + 1) } & \text{if }i \neq j, \\
		\frac{ (d_{n,i} + \theta) (d_{n,j} + \theta + 1) }{ ( 2m + n \theta) (2m + n \theta + 1) } & \text{if }i = j.
	\end{dcases} 
\label{eq-growth}
\end{align}
Note that by this construction, both loops and multiple edges are allowed in $(H_n(m))_{m \geq 0}$, and for each $m \geq 0$, there are $2m$ half-edges in $H_n(m)$.
For each $m \geq 0$, let $D^*_n(m) = (D^*_{n,1}(m), \dots, D_{n,n}^*(m))$ be the degree sequence of $H_n(m)$. 

For $x \in \IR$ and $n \in \IN_0$, write $(x)_n = x (x - 1) \cdots (x - n + 1)$ as the falling factorial and write $x^{(n)} = x (x + 1) \dots (x + n - 1)$ as the rising factorial; the value of each is taken to be 1 if $n = 0$. 
The following lemma states that, conditional on the degree sequence, the random multigraph $H_n(m)$ has distribution $\CM (d_n)$. 
\begin{lemma}
\label{lem7}
Let $d_n = (d_{n,1}, \dots, d_{n,n})$ be a degree sequence satisfying that $\sum_{i = 1}^n d_{n,i} = 2m$.
Then, we have 
\begin{equ}
	\law \bclr{H_n(m) \given D^*_n(m) = d_n } = \CM(d_n). 
    \label{eq-proH}
\end{equ}
\end{lemma}
\begin{proof}[Proof of \cref{lem7}]
Let $G = (x_{ij})_{1 \leq i \leq j \leq n} \in \cM_n$ be a multigraph with the given degree sequence $d_n$.
It follows from \cite[Eqs.\;(2.1)~and~(2.13)]{Pit10} that 
\begin{equ}
	\label{lab40}
	\IP \bigl[ H_n(m) = G \bigr] = \frac{ \prod_{i = 1 }^n \theta^{(d_{n,i})} }{ (n \theta)^{(2m)} } \frac{ (2m)!! }{\prod_{1 \leq i < j \leq n } x_{ij }! \prod_{i = 1}^n x_{ii}!!}, 
\end{equ}
and 
\begin{equation}
	\label{lab41}
	\IP \bigl[ D^*_n(m) = d_n \bigr] = \frac{ (2m)! }{( n\theta)^{(2m)}} \prod_{i = 1 }^n \frac{ \theta^{(d_{n,i})} }{d_{n, i}!}.
\end{equation}
It can be shown (see, e.g., Lemma 1.6 of \cite{Bord06b}) that 
\begin{equ}
  \CM(d_n)\{G\} 
	& = \frac{1}{(2m - 1)!!}\frac{ \prod_{i = 1 }^n d_{n, i}! }{\prod_{1 \leq i < j \leq n } x_{ij }! \prod_{i = 1}^n x_{ii}!!}. 
    \label{eq-proG}
\end{equ}
This completes the proof by combining \cref{lab40,lab41,eq-proG}.  
\end{proof}

Now, we proceed to define an $\cM$-valued stochastic process $(G_n(m))_{m \geq 0}$ which is built on the ideas of $(H_n(m))_{m \geq 0}$. 
For each $m \geq 0$, let $D_n(m) = (D_{n, 1}(m), \dots, D_{n, n}(m) )$ be the degree sequence of $G_{n}(m)$ and let $L_n(m) = \sum_{i = 1 }^n D_{n, i}(m)$. For each $m \geq 1$, we consider the following three types of updates: 
\begin{enumerate}[(I)]

	\item \emph{Add one edge.} In this step, we choose two vertices at random and add an edge between them. Formally, given the graph $G_n(m - 1)$, add one edge between $i$ and $j$ with probability
	\begin{align}
	\begin{dcases}
		\frac{ 2 (D_{n, i}( m - 1 ) + \theta ) (D_{n, j}( m - 1 ) + \theta ) }{ ( L_n({m - 1 }) + n \theta) ( L_n({m - 1 }) + 1 + n \theta) }, & i \neq j, \\
		\frac{ (D_{n, i}( m - 1 ) + \theta) (D_{n, i}( m - 1 ) + \theta + 1) }{ ( L_n(m - 1 ) + n \theta) ( L_n(m - 1 ) + 1 + n \theta) }, & i = j.
	\end{dcases} 
\label{eq-growth2}
	\end{align}
We note that \cref{eq-growth2} is \cref{eq-growth} with $d_{n,i}$ being replaced by the random variable $D_{n,i}(m - 1)$ for each $1 \leq i \leq n$. 
In this step, if $i \neq j$, then the degrees of vertices $i$ and $j$ both increase by 1;  if $i = j$, then the degree of the vertex $i$ increases by 2.

\item \emph{Delete one edge or loop uniformly.} In this step, choose an edge (including loops) uniformly at random and remove it.  If we remove the edge $(i,j)$, then the degrees of vertices $i$ and $j$ both decrease by $1$ and if we remove a loop on
	vertex $i$, then the degree of vertex $i$ decreases by $2$.
\item \emph{Move one half-edge.} In this step, we detach a uniformly chosen half-edge from its vertex and attach it back to another vertex according to a preferential attachment rule.  Formally,  choose 
a half-edge	$j \in [L_n(m - 1 )]$ uniformly at random,  
and let $j'\in [L_n(m - 1 )]$ be the half-edge currently matched with $j$. Then, detach half-edge $j'$ from its vertex and attach it to a new vertex $i$ chosen with probability 
	\begin{align*}
		\frac{ D_{n, i}( m - 1 ) + \theta }{ L_n(m - 1 ) + n \theta }.
	\end{align*}
	If $i \neq j$, then the degree of $i$ increases by $1$ and that of $j$ decreases by $1$; if $i = j$, then $D_{n}(m) = D_n(m - 1)$.
\end{enumerate}
Assume that there exists a positive number $\rho_0 > 0$ such that $L_{n}(0) / n^2 \to \rho_0$ in probability as $n \to \infty$.
Let $a$ be a constant such that $0 < a \leq \rho_0 < \infty$  and 	let $p_1, p_2 \in [ 0, 1 ]$ such that $1 - p_1 - p_2 \geq 0$. 
Let $ (G_{n}(m))_{m \in \IN_0}$ be defined by the following dynamics. Start with $G_{n}( 0 )$ having distribution  $H(L_n(0) / 2 )$. For $m\geq 1$ and given the graph $G_{n}( m - 1)$, do the following:
		\begin{itemize}
			\item If $L_n(m - 1) > an^2 + 1$, 
				generate $G_{n}(m)$ by Step (I) with probability $p_1$, via Step (II) with probability $p_2$ and via Step (III) with probability $1 - p_1 - p_2$;
			\item If $L_n(m - 1) \leq an^2 + 1$, 
				generate $G_{n}(m)$ via Step (I) with probability $p_1 + p_2$ and via Step (III) with probability $1 - p_1 - p_2$.
		\end{itemize}
Therefore, we obtain a sequence of multigraphs $(G_{n}(m))_{m \geq 0}$, which we call the \emph{edge reconnection model}. The following lemma says that $(G_n(m))_{m \geq 0}$ is a multigraph-valued Markov chain with the property that, for each $m \geq 0$ and given $L_n(m) = \ell$, the multigraph $G_n(m)$ has the same distribution as $H_n(\ell/2 )$. 

\begin{lemma}
	\label{lem3.5}
	For each $m \geq 0$ 
	and any even integer $\ell$, we have 
	\begin{align*}
		\law \bigl( G_n (m)\given L_n(m) = \ell\bigr) = \law\bigl( H_n(\ell/2) \bigr).
	\end{align*}
\end{lemma}
\begin{proof}
	[Proof of \cref{lem3.5}]
	Let $G = (x_{ij})_{i,j \in [n]}$ be a nonrandom multigraph
	with degree sequence $d_n = (d_{n,1}, \dots, d_{n,n})$ satisfying that $\sum_{i = 1}^n d_{n,i} = \ell$. 
Recalling \cref{lab40}, it suffices to prove the identity
	\begin{equation}
		\IP \bcls{G_n(m) = G \given L_n(m) = \ell }
		= \frac{ \prod_{i = 1 }^n \theta^{(d_{n,i})} }{ (n \theta)^{(\ell)} } \frac{  (\ell)!! }{\prod_{1 \leq i < j \leq n } x_{ij }! \prod_{i = 1}^n x_{ii}!!},
		\label{eq-l3.5-a}
	\end{equation}
which we prove by induction. The identity is trivial for $m = 0$, which proves the base case. For $m \geq 1$, assume that \cref{eq-l3.5-a} holds for $m - 1$. Assume that $\ell > an^2 + 1$, the other case being similar. Denote by $A_1$, $A_2$ and $A_3$, respectively, the events that $G_n(m)$ is obtained from $G_n(m)$ via Steps (I), (II) and (III), respectively. By the construction of $G_n(m)$, we have 
	\begin{align*}
		\begin{split}
			\MoveEqLeft[1] \IP \bcls{ G_n(m) = G \given L_n(m) = \ell }\\
			& = p_1 \IP \bcls{ G_n(m) = G \given L_n(m) = \ell, A_1 } + p_2 \IP \bcls{  G_n(m) = G \given L_n(m) = \ell, A_2 } \\
			& \quad + (1 - p_1 - p_2) \IP \bcls{  G_n(m) = G \given L_n(m) = \ell, A_3 }.
		\end{split}
	\end{align*}
	Now, given $(i,j)$, 
	let 
	\begin{math}
		G^{(i,j)}_{-} 
	\end{math}
	be the multigraph that is generated by replacing $x_{ij}$ in $G$ by $x_{ij} - 1$ for $i \neq j$ and by replacing $x_{ij}$ by $x_{ij} - 2$ if $i = j$. 
	We have 
	\begin{align*}
		\begin{split}
			\MoveEqLeft \IP \bcls{ G_n(m) = G \given L_n(m) = \ell, A_1 }\\
			& = \IP \bcls{ G_n(m) = G \given L_n(m - 1) = \ell - 2, A_1 }\\
			& = 
			\begin{multlined}[t]
				\sum_{i \leq j} \I \{ x_{ij} \geq 1 \}\IP \bcls{ G_n(m) = G \given G_n(m - 1) = G^{(i,j)}_{-}, A_1 } \\
				 \times \IP \bcls{ G_n(m - 1) = G^{(i,j)}_{-} \given L_n(m - 1) = \ell - 2, A_1 } .
			\end{multlined}
		\end{split}
	\end{align*}
	Observe that  for any $(i,j)$ such that $x_{ij} \geq 1$, 
	\begin{align*}
		\MoveEqLeft \IP \bcls{ G_n(m) = G \given G_n(m - 1) = G^{(i,j)}_{-} , A_1} \\
		& = 
		\begin{dcases}
			\frac{ 2 (d_{n,i} + \theta - 1)(d_{n,j} + \theta - 1) }{(\ell - 2 + n\theta)(\ell - 1 + n\theta)} & \text{if $i \neq j$},	\\
			\frac{  (d_{n,i} + \theta - 2)(d_{n,i} + \theta - 1) }{(\ell - 2 + n\theta)(\ell - 1 + n\theta)} & \text{if $i = j$}.
		\end{dcases}
	\end{align*}
	By induction assumption, noting that $A_1$ is independent of $(G_{n}(m - 1), L_n(m - 1))$, we obtain
	\begin{align*}
		\MoveEqLeft[1] \IP \bcls{ G_n(m - 1) = G^{(i,j)}_{-} \given L_n(m - 1) = \ell - 2, A_1 } \\
		& =  \IP \bcls{ G_n(m - 1) = G^{(i,j)}_{-} \given L_n(m - 1) = \ell - 2 } \\
		& = 
		\begin{multlined}[t]
			\biggl(\frac{ \prod_{i = 1 }^n \theta^{(d_{n,i})} }{ (n \theta)^{(\ell)} } \frac{ (\ell)!! }{\prod_{1 \leq i < j \leq n } x_{ij }! \prod_{i = 1}^n x_{ii}!!} \biggr) \times \frac{(n\theta + \ell - 2)(n\theta + \ell - 1)}{\ell}\\
			\times  
			\begin{dcases}
				\frac{ \theta^{(d_{n,i} - 1)}\theta^{(d_{n,j} - 1)} }{\theta^{(d_{n,i})}\theta^{(d_{n,j})}} \frac{x_{ij}!}{(x_{ij} - 1)!}  & \text{if $i \neq j$}, \\
				\frac{ \theta^{(d_{n,i} - 2)} }{\theta^{(d_{n,i})}} \frac{x_{ii}!!}{(x_{ii} - 2)!!}  & \text{if $i = j$}.
			\end{dcases}
		\end{multlined}
	\end{align*}
	Then, it follows that 
	\begin{align*}
		\begin{split}
			\MoveEqLeft\IP \bcls{ G_n (m) = G \given L_n(m) = \ell, A_1 } \\
			& = \frac{1}{\ell}\biggl(2 \sum_{i <  j} x_{ij} + \sum_{i = 1}^n x_{ii}\biggr)  \times \biggl(\frac{ \prod_{i = 1 }^n \theta^{(d_{n,i})} }{ (n \theta)^{(\ell)} } \frac{ (\ell)!! }{\prod_{1 \leq i < j \leq n } x_{ij }! \prod_{i = 1}^n x_{ii}!!} \biggr)
			\\
			& = \frac{ \prod_{i = 1 }^n \theta^{(d_{n,i})} }{ (n \theta)^{(\ell)} } \frac{ (\ell)!! }{\prod_{1 \leq i < j \leq n } x_{ij }! \prod_{i = 1}^n x_{ii}!!} .
		\end{split}
	\end{align*}
	Given $(i,j)$, let $G^{(i,j)}_{+}$ be the multigraph generated by replacing $x_{ij}$ in $G$ by $x_{ij} + 1$ if $i \neq j$ and by replacing $x_{ij}$ by $x_{ij} + 2$ if $i = j$. Then, 
	\begin{align*}
		\MoveEqLeft \IP \bcls{ G_{n}(m) = G \given L_n(m) = \ell, A_2 }\\
		& = 
		\begin{multlined}[t]
			\sum_{i \leq j} \IP \bcls{ G_n(m) = G \given G_{n}(m - 1) = G^{(i,j)}_{+}, A_2 }  \\
			\times \IP \bcls{ G_{n}(m - 1) = G^{(i,j)}_{+} \given L_n(m - 1) = \ell + 2, A_2 }.
		\end{multlined}
	\end{align*}
	Now, 
	\begin{align*}
		\IP \bcls{ G_n(m) = G \given G_{n}(m - 1) = G^{(i,j)}_{+} , A_2} = 
		\begin{dcases}
			\frac{2 (x_{ij} + 1)}{(\ell + 2)} & \text{if $i \neq j$},\\
			\frac{ (x_{ij} + 2)}{(\ell + 2)} & \text{if $i = j$},
		\end{dcases}
	\end{align*}
	and 
	\begin{equation*}
		\begin{split}
			\MoveEqLeft
			\IP \bcls{ G_{n}(m - 1) = G^{(i,j)}_{+} \given L_n(m - 1) = \ell + 2, A_2 } \\
			& = 
			\begin{multlined}[t]
				\biggl(\frac{ \prod_{i = 1 }^n \theta^{(d_{n,i})} }{ (n \theta)^{(\ell)} } \frac{ (\ell)!! }{\prod_{1 \leq i < j \leq n } x_{ij }! \prod_{i = 1}^n x_{ii}!!} \biggr) \times \frac{(\ell + 2)}{(n\theta + \ell)(n\theta + \ell + 1)}\\
				\quad \times 
				\begin{dcases}
					\frac{ \theta^{(d_{n,i} + 1)}\theta^{(d_{n,j} + 1)} }{\theta^{(d_{n,i})}\theta^{(d_{n,j})}} \frac{x_{ij}!}{(x_{ij} + 1)!}  & \text{if $i \neq j$}, \\
					\frac{ \theta^{(d_{n,i} + 2)} }{\theta^{(d_{n,i})}} \frac{x_{ii}!!}{(x_{ii} + 2)!!}  & \text{if $i = j$}.	
				\end{dcases}
			\end{multlined}
		\end{split}
	\end{equation*}
	Then, it follows that
	\begin{align*}
		\IP \bcls{ G_n(m) = G \given L_n(m) = \ell, A_2 } = \biggl(\frac{ \prod_{i = 1 }^n \theta^{(d_{n,i})} }{ (n \theta)^{(\ell)} } \frac{ (\ell)!! }{\prod_{1 \leq i < j \leq n } x_{ij }! \prod_{i = 1}^n x_{ii}!!} \biggr). 
	\end{align*}

	Using a similar argument, we have 
	\begin{align*}
		\IP \bcls{ G_n(m) = G \given L_n(m) = \ell, A_3 } = \biggl(\frac{ \prod_{i = 1 }^n \theta^{(d_{n,i})} }{ (n \theta)^{(\ell)} } \frac{ (\ell)!! }{\prod_{1 \leq i < j \leq n } x_{ij }! \prod_{i = 1}^n x_{ii}!!} \biggr).
	\end{align*}
	Combining the foregoing inequalities, we conclude that \cref{eq-l3.5-a} also holds for $m$. This completes the proof by induction. 
\end{proof}

The limiting behavior of the edge reconnection model was firstly studied by \cite[]{Rat12} who defined the dynamics only based on Step (II). In that case, the total number of the edges does not change over time. We remark that the model $(G_n(m))$ in the present
paper is more general. Specially, if $p_1 = p_2 = 0$, then our model reduces to \cite[]{Rat12}'s model. 
In what follows, we proceed to prove the scaled multigraphon process converges in distribution to a non-trivial multigraphon-valued limiting process. 

		Let $\kappa_n = (\kappa_{n}(s))_{s \geq 0} \in \cD$, where for each $s \geq 0$,  multigraphon $\kappa_n(s)$ is the corresponding multigraphon generated by a scaled process $G_{n}(\lfloor n^4 p_1^{-1 } s \rfloor)$. 
Let $Y_n(s) = L_n( \lfloor n^4 p_1^{-1} s \rfloor )/n^2$ and $Y_n = (Y_n(s))_{s \geq 0}$. 
		In order to specify the limiting multigraphon process of $\kappa_n$, we need to introduce the limiting process of $Y_n$.

	Let $Y = (Y(s))_{s \geq 0}$  be defined as 
	\begin{align}
		Y(s) = a  + |2 B(s) + \rho_0 - a| \quad \text{for $s \geq 0$},
		\label{eq-limY}
	\end{align}
	where $ (B(s))_{s \geq 0}$ is a standard Brownian motion. 
	Recalling that $\theta$ is given as in \cref{eq-growth},
	let 
	\begin{align}
		\Psi(x) & =  
		\begin{cases}
			\frac{\theta^{\theta}}{\Gamma(\theta)}\int_0^x  z^{\theta - 1} e^{- \theta z} dz & \text{if $x \geq 0$}, \\
			0 & \text{otherwise}.
		\end{cases}\label{eq-Psi}\\
		\bar\Psi(x) & = \inf \{ y: \Psi(y) \geq x  \}. 
		\label{eq-Psi-inv}
	\end{align}
	Then, it follows that $\bar\Psi(x)$ is the general inverse function of $\Psi (x)$ with respect to $x$. 
	Let $\kappa = (\kappa(s))_{s \geq 0} \in \cD$ be a multigraphon process such that for $s \geq 0$ and $r \in \IN_0$,  
\begin{equation}
	\label{eq17}
	\kappa (s; r; x, y) = 
	\begin{dcases*}
		{p} \bigl( r;  Y(s)\bar\Psi (x) \bar\Psi (y)  \bigr) & if $x \neq y$, \\
		{p} \biggl( \frac{r}{2}; \frac{ Y(s)\bar\Psi (  x) \bar\Psi (y) }{2} \biggr) & if $x = y$ and if $r$ is even, \\
		0 & otherwise,
	\end{dcases*}
\end{equation}
where ${p}(r;\lambda) = \lambda^{r} e^{-\lambda} / r! $ as before. We have the following result.

\begin{theorem}
	\label{thm5}
	Assume that $p_1 = p_2 > 0$. Then,
	\begin{math}
		\kappa_n \wconv \kappa 
	\end{math} 
	in $(\cD, \dcir)$.
\end{theorem}
\begin{remark}
	When $p_1 \neq p_2$, we need to use a different time-scaling for $G_n(m)$.
	If $p_1 > p_2$, we have by the law of large numbers that $L_n( \lfloor n^2 s \rfloor ) / n^2$ diverges to $\infty$ in probability as both $n$ and $s$ tend to infinity. As a result, the multigraph $G_{n}( \lfloor n^2 s\rfloor )$ diverges to a
	multigraph with infinite edges and infinite loops as $n, s \to \infty$. If, on the other hand, $p_1 < p_2$, then $L_n( \lfloor n^2 s \rfloor ) / n^2$ converges to $a$ in probability as $n$ and $s$ go to infinity, where $a$ is as in the generation of $(G_n(m))_{m \geq 1}$. Consequently, as $n$ and $s$ tend to infinity, by \cref{t3.1}, the limiting multigraphon of $G_{n}( \lfloor n^2
	s\rfloor )$ is a nonrandom multigraph given by 
	\begin{align*}
		h(r; x, y) = 
		\begin{dcases}
			p(r; a \bar \Psi(x) \bar \Psi(y)) & \text{if $x \neq y$}, \\
			p \biggl( \frac{r}{2}; \frac{a \bar \Psi(x) \bar \Psi(y)}{2}\biggr) & \text{if $x = y$ and $r$ is even}, \\
			0 & \text{otherwise}.
		\end{dcases}
	\end{align*}
\end{remark}

Before giving the proof of \cref{thm5}, we  introduce some notation and prove some auxiliary results. 
Let $f : \IR^k
\to \IR$ be a measurable function and we say $f$ is symmetric if  $f(x_1, \dots, x_k) = f(x_{\sigma(1)}, \dots, x_{\sigma(k)})$ for any $(x_1, \dots, x_k) \in \IR^k$ and $\sigma : [k] \hookrightarrow [k]$. 
For any symmetric function $f$ and $x = (x_1, \dots, x_n)$,
define the $U$-statistic 
\begin{equation}
	\label{eq-Ufx}
	U_f(x) = \frac{1}{(n)_k} \sum_{\sigma : [k] \hookrightarrow [n]} f(x_{\sigma(1)}, \dots, x_{\sigma(k)} ).  
\end{equation}
We have the following concentration inequality result. 
\begin{lemma}
	\label{lem-11}
	Let $m$ be any positive integer satisfying that $ n \leq m \leq n^{3} $, let $H(m)$ be defined as above and let $D^*(m) = (D_1^*(m),\dots,D_n^*(m))$ be its degree sequence. 
	Let $Z = (Z_1, \dots, Z_n)$ be a vector of independent random variables with the common negative binomial distribution $\mathrm{NB}(\theta, 2m/(2m + n\theta))$, that is, the probability mass function is given by 
\begin{align*}
	\IP [ Z_1 = r ] = \biggl( \frac{n \theta}{2m + n\theta} \biggr)^{\theta} \biggl( \frac{2m}{2m + n \theta} \biggr)^{r} \frac{ \theta^{( r )} }{r!}, \quad r \in \IN_0. 
\end{align*}
	Then there exist positive constants $C$ and $C'$ that depend only on $k $ and $\theta$, such that, for any symmetric function $f : \IR^k \to \IR$ with $0 \leq f \leq 1$ and any $\eps > 0$, we have 
	\begin{equation}
		\IP \bigl[ \bigl\lvert U_f( D^*(m) ) - \IE \{ f(Z_{1}, \dots, Z_k) \} \bigr\rvert \geq \eps  \bigr] 
		\leq C n^{5/2} e^{ - C' n \eps^2 }
		\label{l11-a}
	\end{equation}	
\end{lemma}
\begin{proof}
	[Proof of \cref{lem-11}]
	Let $\IP^*$ and $\IE^*$ denote  probability  and expectation conditional on the event that $\sum_{i = 1}^n Z_i = 2m$.
	Let $C_1, C_2, \dots$ denote positive constants depending only on $k$ and $\theta$. 
	It has been shown that (see \cite[p. 624]{Pit10})
	\begin{equation}
		\law(D^*_1(m), \dots, D^*_n(m)) = \law \bbbclr{ Z_1, \dots, Z_n \given \sum_{i = 1}^n Z_i = 2m }. 
        \label{11l-c}
	\end{equation}	
	By definition, we have 
	\begin{equation}
		\begin{split}
			\IP \bbbcls{ \sum_{i = 1}^n Z_i = 2m }
			& = \sum_{z_1 + \dots + z_n = 2m} \IP \bcls{ Z_1 = z_1, \dots, Z_n = z_n } \\
			& = \biggl( \frac{n\theta}{2m + n\theta} \biggr)^{n\theta} \biggl(   \frac{2m}{2m + n\theta}\biggr)^{2m} \frac{ (n\theta)^{(2m)} }{(2m)!}\\
			& = \frac{1}{2m} \bbbclr{ \frac{n\theta}{2m + n\theta} }^{n\theta} \bbbclr{   \frac{2m}{2m + n\theta}}^{2m} \frac{ \Gamma(n\theta + 2m)}{\Gamma(n \theta) \Gamma(2m)}\\
			& \geq \frac{ (n \theta)^{1/2} }{2m + n \theta} \exp\Bigl({ -\frac{1}{12n\theta} - \frac{1}{24m} }\Bigr)  \\
			& \geq C_1 n^{-5/2}, 
		\end{split}
		\label{11l-2}
	\end{equation}
	where we used \cref{11l-0} and the fact that $n \leq m(n) \leq n^{3}$ in the last two lines. 
	By \cref{11l-c,11l-2}, the left hand side of \cref{l11-a} becomes 
	\begin{equation}
		\begin{split}
			\IP^* \bcls{ \bigl\lvert U_f(Z) - \IE \{ U_f(Z) \} \bigr\rvert \geq \eps } 	
			& = \frac{\IP \bcls{ \bigl\lvert U_f(Z) - \IE \{ U_f(Z) \} \bigr\rvert \geq \eps }}{\IP [\sum_{i = 1}^n Z_i = 2m]} \\
			& \leq C_2 n^{5/2} \IP \bcls{ \bigl\lvert U_f(Z) - \IE \{ U_f(Z) \} \bigr\rvert \geq \eps } .
		\end{split}
		\label{11l-1}
	\end{equation}
	As  $0 \leq f \leq 1$, the value of $U_f(Z)$ changes by at most $(n - 1)_{k - 1}/(n)_k$ if the $i$-th variable $Z_i$ changes. Recalling that $Z_1, \dots, Z_n$ are independent,
	by the McDiarmid inequality, we have 
	\begin{equation*}
		\IP \bcls{ \bigl\lvert U_f(Z) - \IE \{ U_f(Z) \} \bigr\rvert \geq \eps }
		\leq 2 \exp \bigl( -C_3{ n \eps^2} \bigr).
	\end{equation*}
	This completes the proof together with \cref{11l-1}. 
\end{proof}

The following lemma provides a general concentration inequality for graph functionals.
For any two multigraphs $G, G'\in \mathcal{M}_n$, we say $G$ and $G'$ differ from each other by a single switch of edges, if $G'$ is a multigraph generated by choosing two edges or loops from $G$ and reconnecting these four half-edges.
\begin{lemma}[Remark 3.31 of \cite{Bord06b}]
	\label{lem5}
	Let $d_n$ be a degree sequence, 
	let $G_n \sim \CM (d_n)$ and let $f: \cM_n \to \IR$ be a measurable function. Assume that there exists $c_1 > 0$ such that  
	\begin{align*}
		\bigl\lvert  f(G) -f(G') \bigr\rvert \leq c_1 
	\end{align*}
	for any multigraphs $G, G'\in \mathcal{M}_n$  differing from each other by a single switch of edges. 
	Then, for any $\eps \geq 0$,  
	\begin{align*}
		\IP \bcls{ \bigl\lvert f(G_n) - \IE f(G_n) \bigr\rvert \geq \eps  } \leq 2 \exp \biggl(- \frac{\eps^2}{c_1^2 \sum_{i = 1}^{n} d_{n,i}} \biggr)  .
	\end{align*}
\end{lemma}

\begin{lemma}
	\label{lem6}
	We have for each $F \in \cM_k$, $m \geq 0$ and $\eps \geq 0$,  
	\begin{multline}
		\label{label34}
		\IP \bcls{ \lvert t_F^{\ind} (G_n(m)) - \IE \{ t_F^{\ind} (G_n (m)) \vert D_n(m) \} \rvert \geq \eps \given L_n(m) \leq n^3 } \\
		\leq 2 \exp \bigl( - {  C n \eps^2 } \bigr)  ,
	\end{multline}
	where $C > 0$ is a constant depending on $k$. 
\end{lemma}
\begin{proof}
	Let $G, G' \in \cM_n$ be two multigraphs differ from each other by a single switch of edges or loops. By definition, for any $F \in \cM_k$, 
	\begin{align*}
		\bigl\lvert t_F^{\ind} (G) - t_F^{\ind} (G') \bigr\rvert \leq \frac{C_1}{n (n - 1)},
	\end{align*}
	where $C_1 > 0$ is a constant depending only on $k$. 
	By \cref{lem7}, for each $m \geq 0$ and given $D_n(m) = d_n$, the random multigraph $G_n(m)$ has the same distribution as $\CM(d_n)$, and then \cref{lem5} implies \cref{label34}, as desired.
\end{proof}

\begin{lemma}
	\label{lem-3.10}
	Recall that $\theta$ is as defined in \cref{eq-Psi} and $a$ is as in the construction of $G_n(m)$. Let $k \geq 1$, $y \geq a$ and let $Z_{n,1}, \dots, Z_{n,k}$ be independent random variables with the common negative binomial distribution $\mathrm{NB}(\theta, ny/(ny
	+ \theta))$, that is, the probability mass function is given by 
\begin{equ}
	\IP [ Z_{n,j} = r ] = \biggl( \frac{\theta}{ny + \theta} \biggr)^{\theta} \biggl( \frac{ n y }{ny + \theta} \biggr)^{r} \frac{ \theta^{( r )} }{r!}, \quad 1 \leq j\leq k,\ r \in \IN_0.
    \label{eq-Zmass}
\end{equ}
Let $\zeta_{n,j} = Z_{n,j}/(n y)$ for every $1 \leq j \leq k$, and let $\zeta_{1}, \dots, \zeta_k$ be independent random variables with the common distribution function \cref{eq-Psi}. 
Let $\phi : \IR^k \to \IR$ be a bounded measurable function satisfying that there exists $c > 0$ such that $\lVert \phi \rVert \leq c$. We have 
\begin{equ}
	|\IE \phi(\zeta_{n,1}, \dots, \zeta_{n,k}) - \IE \phi(\zeta_1, \dots, \zeta_k)| \leq C n^{-1/2},
    \label{eq-Ephi}
\end{equ}
where $C > 0$ is a constant depending only on $a,k,c$ and $\theta$.

\end{lemma}
\begin{proof}
This proof includes two parts. In the first part, we prove an approximate representation of \cref{eq-Zmass}, and in the second part, we prove \cref{eq-Ephi}.

Denote by $C$ a general constant depending only on $a, k ,c$ and $\theta$, which might take different values in different places. 
Letting $u_n(r) = r/(ny)$, we have the probability mass function of $Z_{n,1}$ can be rewritten as 
\begin{equ}
	\IP [ Z_{n,1} = r ]
	& = \theta^{\theta} ( {n y + \theta} )^{-\theta} \bigl( 1 + \theta/(n y) \bigr)^{ - ny u_n(r) } \frac{ \theta^{( r )} }{ r! } .
    \label{eq-Zmass1}
\end{equ}
For the second factor of the right hand side of \cref{eq-Zmass1}, noting that $y \geq a$, we have 
\begin{equ}
	(ny + \theta)^{-\theta} 
	& = (ny)^{-\theta} \biggl( 1 + \frac{\theta}{ny} \biggr)^{-\theta}\leq (ny)^{-\theta} (1 + Q_0)
    \label{eq-l3.10-0}
\end{equ}
for some $Q_0 \leq C n^{-1}$. 
For the third factor of the right hand side of \cref{eq-Zmass1}, noting that $y \geq a$, we have if $u_n(r) \geq n^{-1/2}$,
\begin{equ}
	\bigl( 1 + \theta/(n y) \bigr)^{ - ny u_n(r) } 
	\leq e^{- \theta u_n(r) } (1 + Q_1)
    \label{eq-l3.10-1}
\end{equ}
for some $\lvert Q_1 \rvert \leq C n^{-1/2}$. 
For the last factor of the right hand side of \cref{eq-Zmass1}, we obtain if $u_n(r) \geq n^{-1/2}$, then $r = nyu_n(r) \geq n^{-1/2}a$, and therefore, by Stirling's formula \cref{11l-0} again, 
\begin{equ}
	\frac{ \theta^{( r )} }{  r ! }
	& = \frac{ \Gamma(r + \theta ) }{\Gamma(\theta) r!} \leq \frac{(r + \theta)^{\theta - 1}}{\Gamma(\theta)}		(1 + Q_2) 
    \label{eq-l3.10-2}
\end{equ}
for some $| Q_2 | \leq C n^{-1/2}$.
Moreover, if $u_n(r) \geq n^{-1/2}$, 
we have  
\begin{equ}
	(r + \theta)^{\theta - 1} 
	& = (n y u_n(r) + \theta)^{\theta - 1}\\
	& = (ny)^{\theta - 1} u_n(r)^{\theta - 1} \biggl( 1 + \frac{\theta}{nyu_n(r)} \biggr)^{\theta - 1}\\
	& \leq (ny)^{\theta - 1} u_n(r)	^{\theta - 1} (1 + Q_3),
    \label{eq-l3.10-3}
\end{equ}
for some $\lvert Q_3 \rvert \leq C n^{-1/2}$. 
Substituting \cref{eq-l3.10-0,eq-l3.10-1,eq-l3.10-2,eq-l3.10-3} to \cref{eq-Zmass1}, we have if $u_n(r) \geq n^{-1/2}$,
\begin{align*}
	\IP [ Z_{n,1} =  r] & \leq \frac{ \theta^\theta }{ ny \Gamma(\theta)} u_n(r)^{\theta - 1} e^{-\theta u_n(r)} (1 + e^{C n^{-1/2}}). 
\end{align*}
A similar lower bound still holds. 
Hence, it follows that if $u_n(r) \geq n^{-1/2}$, we have 
\begin{equ}
	\IP [ Z_{n,1} =  r] & = \frac{ \theta^\theta }{ ny \Gamma(\theta)} u_n(r)^{\theta - 1} e^{-\theta u_n(r)} (1 + Q_3	),
    \label{eq-l3.10-4}
\end{equ}
for some 
\begin{math}
	\lvert Q_3 \rvert \leq C n^{-1/2}.	
\end{math}
Similarly, if $0 \leq u_n(r) \leq n^{-1/2}$, we have 
\begin{equ}
	\IP [ Z_{n,1} =  r] & \leq C (ny)^{-1}.
    \label{eq-l3.10-5}
\end{equ}

Now, we apply \cref{eq-l3.10-4,eq-l3.10-5} to prove \cref{eq-Ephi}.
Recalling that $\zeta_{n,j} = Z_{n,j}/(ny)$, denote by $A_{n}$ the event that $\{\zeta_{n,j} \geq n^{-1/2} \text{ for $1 \leq j \leq k$} \}$ and by $B_n$ the event that $\{  \zeta_{j} \geq n^{-1/2} \text{ for $1 \leq j
\leq k$} \}$. By \cref{eq-l3.10-4} and recalling again that $ny \geq na$, we have 
\begin{align}
	\begin{split}
		\MoveEqLeft[1]
		\IE \{ \phi( \zeta_{n,1}, \dots, \zeta_{n,k} ) \I ( A_n ) \}\\
		& = \frac{\theta^{k\theta}}{(ny)^k\Gamma(\theta)^k} \sum_{n^{1/2} y \leq r_1 \leq \infty} \dots \sum_{n^{1/2} y \leq r_k \leq \infty} \\
		& \hspace{2cm} \times \biggl( \phi \Bigl( \frac{r_1}{ny}, \dots, \frac{r_k}{ny} \Bigr) \IP \bigl[ Z_{n,1} = r_1 , \dots, Z_{n,k} = r_k \bigr] \biggr) \\
		& = \frac{(1 + Q_4)\theta^{k\theta}}{\Gamma(\theta)^k}\int_{ [n^{-1/2},\infty]^k} \phi(u_1, \dots, u_k) \prod_{j = 1}^k (u_j^{\theta - 1} e^{-\theta u_j} d u_j)  \\
		& = \IE \{ \phi(\zeta_1, \dots, \zeta_k) \I(B_n) \} (1 + Q_5), 
	\end{split}
	\label{eq-l3.10-6}
\end{align}
for some $\lvert Q_4 \rvert \leq C n^{-1/2}$ and $\lvert Q_5 \rvert \leq C n^{-1/2}$. 
On the event $A_n^c$, we have 
\begin{equ}	
	\bigl\lvert \IE \{ \phi( \zeta_{n,1}, \dots, \zeta_{n,k} ) \I ( A_n^c ) \} \bigr\rvert
	& \leq c \IP \bbcls{ \min_{1 \leq j \leq k} \zeta_{n,j} < n^{-1/2} } \\
	& \leq c \sum_{j = 1}^k  \IP \bcls{ \zeta_{n,j} < n^{-1/2} }\\
	& \leq C n^{-1/2} , 
    \label{eq-l3.10-6'}
\end{equ}
where we used \cref{eq-l3.10-5} in the last line. 
Since $\zeta_1, \dots, \zeta_k$ follow the common Gamma distribution $\Gamma(\theta,\theta)$, using a similar argument, we have 
\begin{equ}
	|\IE \{ \phi(\zeta_1,\dots,\zeta_k) \I( B_n^c ) \}| \leq C n^{-1/2}.
    \label{eq-l3.10-9}
\end{equ}
Combining \cref{eq-l3.10-6,eq-l3.10-6',eq-l3.10-9}, we complete the proof. 
\end{proof}

We are now ready to give the proof of \cref{thm5}. 
\begin{proof}
[Proof of \cref{thm5}]
We use \emph{(v)} of \cref{cor1} to prove this result. 
The proof is separated into three parts. We first show that $Y_n$ is weakly convergent, then we verify the tightness property of $( t_F^{\ind}(\kappa_n) )_{n \geq 1}$, and finally, we prove the finite dimensional convergence of $(t_{F_1}^{\ind}(\kappa_n(s_1)), \dots, t_{F_q}^{\ind}
(\kappa_n(s_q)))$. 

{\medskip\noindent \it Step 1. Weak convergence of $Y_n$.}
Let $\xi_1, \xi_2, \dots$ be i.i.d.\ random variables with common probability distribution 
\begin{align*}
	\IP (\xi_1 = 1) = \IP (\xi_1 = -1) = p_1, \quad \IP(\xi_1 = 0) = 1 - 2 p_1. 
\end{align*}
Let $S(m) = \xi_1 + \dots + \xi_m$. 
Note that 
\begin{align*}
	L_n(m) \stackrel{d}{=} an^2 + \lvert 2 S(m) + L_n(0) - an^2 \rvert, 
\end{align*}
and that $L_{n}(0)/n^2 \to \rho_0$ as $n \to \infty$, 
and then we have  
\begin{align*}
	Y_n(s) \stackrel{d}{=} a + \lvert 2 n^{-2} S( \lfloor n^4 p_1^{-1} s \rfloor ) + \rho_0 - a \rvert. 
\end{align*}
Now, as $(n^{-2} S( \lfloor n^4 p_1^{-1} s \rfloor) )_{s \geq 0} \wconv (B (s))_{s \geq 0}$ ($n\toinf$), where $(B(s))_{s \geq 0}$ is a standard Brownian motion, and by continuous mapping theorem, we have 
\begin{math}
	Y_n \wconv Y~(n \toinf), 
\end{math}
where $Y$ is as in \cref{eq-limY}. By Skorokhod's representation theorem, we may assume that $Y_n$ and $Y$ are constructed in a probability space $(\Omega, \cF, \IP)$ such that $Y_n \longto Y$ $\IP$-a.s. as $n \toinf$.

Moreover, for $n \geq 4(a + \rho_0)$ and for any $T \geq 0$, we have 
\begin{align*}
	\begin{split}
			\IP \biggl[ \sup_{0 \leq s \leq T} Y_n(s) \geq n \biggr] 
		& \leq \IP \biggl[ \sup_{0 \leq s \leq T} | S( \lfloor n^4 p_1^{-1}s \rfloor ) | \geq \frac{n^3}{2} - (a + \rho_0) n^2 \biggr] \\
		& \leq \IP \biggl[ \sup_{0 \leq s \leq T} \biggl\lvert \sum_{i = 1}^{ \lfloor n^4 p_1^{-1}s \rfloor } \xi_i \biggr\rvert \geq \frac{n^3}{4}\biggr]\\
		& \leq 2 \IP \biggl[ \biggl\lvert \sum_{i = 1}^{ \lfloor n^4 p_1^{-1}T \rfloor } \xi_i \biggr\rvert \geq \frac{n^3}{4}\biggr], 
	\end{split}
\end{align*}
where we used L\'evy's inequality since $\xi_i$'s are symmetric. 
Applying Hoeffding's inequality, we have for $n \geq 4(a + \rho_0)$ and $T \geq 0$, 
\begin{equation}
		\IP \bbbcls{ \sup_{0 \leq s \leq T} Y_n(s) \geq n } 
		 \leq 4 \exp \biggl( - \frac{ 8 n^2 }{p_1^{-1}(T + 1)} \biggr). 
	\label{eq-expY}
\end{equation}
The inequality \cref{eq-expY} will be used in \emph{Step 2}. 

{\medskip\noindent \it Step 2. Tightness of $(t_F^{\ind} (\kappa_n))_{n \geq 1}$.}
Fix $F = (a_{ij})_{i,j \in [k]}$. Let $(\zeta_1, \dots, \zeta_k)$ be a vector of independent and identically distributed random variables (independent of everything else) having the common distribution function \cref{eq-Psi}. For any $y \geq a$, let $Z_{n,1}, \dots, Z_{n,k}$
be independent random variables with the common distribution $\mathrm{NB}(\theta, ny/(ny + \theta))$, and let $\zeta_{n,j} = Z_{n,j}/ (ny)$ for each $1 \leq i \leq k$.
Let 
\begin{align*}
	f(x_1, \dots, x_k ; y) = \prod_{1 \leq i < j \leq k} p(a_{ij}; y x_i x_j) \prod_{i = 1}^k p \biggl( \frac{a_{ij}}{2} ; \frac{ y x_i^2 }{2} \biggr), 
\end{align*}
and let 
\be{
  \psi(y) = \IE f( \zeta_1, \dots, \zeta_k ; y ), \quad \psi_n(y) = \IE f(\zeta_{n,1}, \dots, \zeta_{n,k} ; y). 
}

Write $D'_n(s) = D_n ( \lfloor n^4 p_1^{-1} s \rfloor )$; that is, $D'_n(s)$ is the degree sequence of the multigraph $G_n( \lfloor n^4 p_1^{-1} s\rfloor )$. Clearly, $D'_{n,1}(s) + \dots + D'_{n,n}(s) = n^2 Y_n(s)$. 
Let 
$
	\bar{D}_n(s) =  ( \bar{D}_{n,1}(s), \dots, \bar{D}_{n,n}(s) )$ where
$\bar{D}_{n,i}(s) =  D'_{n,i}(s)/(n Y_n(s))   
$ for $1 \leq i \leq n$. 
Recalling \cref{eq-Ufx}, for each $s \geq 0$, define  
\begin{align*}
	U_n(s) & =  U_{f(\cdot; Y_n(s))} (\bar D_n(s)) \\
		   & = \frac{1}{(n)_k} \sum_{\sigma : [k] \hookrightarrow [n]} f \bigl( \bar D_{n,\sigma(1)}(s), \dots, \bar D_{n, \sigma(k)}(s); Y_n(s)\bigr).
\end{align*}
For $j = 1, 2, 3, 4, 5$, let $W_{j,n}$ be $\IR$-valued stochastic process defined by 
\begin{align*}
	W_{1,n}(s) & =  \psi(Y_n(s)), \\
	W_{2,n}(s) & = \psi_n(Y_n(s)), \\
	W_{3,n}(s) & = U_n(s), \\
	W_{4,n}(s) & = \IE \bclc{ t_{F}^{\ind}( \kappa_n(s) ) \given D_n'(s) }, \\
	W_{5,n}(s) & = t_{F}^{\ind}(\kappa_n(s)).
\end{align*}
As $Y_n \wconv Y$, it follows that $(Y_n)_{n \geq 1}$ is tight. 
Since compact sets remain compact under continuous mappings, it follows that $(W_{1,n})_{n \geq 1}$ is also tight. 
In order to prove the tightness of $(t_{F}^{\ind}(\kappa_n))_{n \geq 1}$, for $1 \leq j \leq 4$, we prove that for any $\eps > 0$ and $T > 0$, 
\begin{equ}
	\limsup_{n \to \infty} \IP \biggl[ \sup_{0 \leq s \leq T} \bigl\lvert W_{j, n}(s) - W_{j + 1, n}(s) \bigr\rvert \geq \eps  \biggr] \leq \eps. 
    \label{eq-Wconv}
\end{equ}
Then, by the tightness of $(W_{1,n})_{n \geq 1}$ and by \cite[Problem 18, p.~152]{Eth05}, we have $(W_{5,n})_{n \geq 1}$ is tight.

Noting that for every $y$, the function $f$
is a bounded function of $(z_1, \dots, z_k)$, and by \cref{lem-3.10} with $\phi = f(\cdot;y)$, 
we have for every $y \geq a$, as $n \to \infty$,
\begin{equation*}
	\sup_{y \geq a} \lvert \psi_n(y) - \psi(y) \rvert \to 0, 
\end{equation*}
which implies 
\begin{equ}
	\sup_{s \geq 0}\lvert W_{1,n}(s) - W_{2,n}(s) \rvert \longto 0 \quad \text{$\IP$-a.s.\ as $n \toinf$.} 
    \label{eq-3.48}
\end{equ}
This proves \cref{eq-Wconv} for $j = 1$.  

Let $\IE^*$ and $\IP^*$ be the expectation operator and probability operator conditional on $(Y_n(s))_{s \geq 0}$.
By \cref{lem-11,eq-expY} and noting that
$\psi_n(Y_n(s)) = \IE^* \{ f ( \zeta_{n,1}, \dots, \zeta_{n,k};
Y_n(s) ) \} $, we have for $n \geq 4(a + \rho_0)$, 
\begin{align*}
	\MoveEqLeft 
	\IP \bbbcls{ \sup_{0 \leq s \leq T} \bigl\lvert W_{2,n}(s) - W_{3,n}(s) \bigr\rvert \geq \eps_1		} \\
	& \leq \IP \bbbcls{ \sup_{0 \leq s \leq T} \bigl\lvert U_n(s) - \psi_n(Y_n(s)) \bigr\rvert \geq \eps_1		} \\
	& 
		\leq  \IP \bbbcls{ \sup_{0 \leq s \leq T} \bigl\lvert U_n(s) - \psi_n(Y_n(s)) \bigr\rvert \geq \eps_1	\biggm\vert \sup_{0 \leq s \leq T} Y_n(s) \leq n	} \\
	& \quad \quad + \IP \bbbcls{ \sup_{0 \leq s \leq T} Y_n(s) > n } 
	\\
	& \leq C_1 n^{5/2} \sum_{m = 0}^{ \lfloor n^{4}p_1^{-1} T \rfloor + 1 } e^{- C_2 n \eps_1^2} + 4 \exp \biggl( - \frac{8n^2}{p_1^{-1}(T + 1)} \biggr)\\
	& \leq C_1 (1 + T) n^{13/2} e^{-C_2 n \eps_1^2} + 4 \exp \biggl( - \frac{8n^2}{p_1^{-1}(T + 1)} \biggr). 
\end{align*}
Then, there exists an $n_1 > 0$ depending on $\eps_1, T, p_1, C_1$ and $C_2$ such that 
\begin{align*}
	\IP \bbbcls{ \sup_{0 \leq s \leq T} \bigl\lvert  W_{2,n}(s) - W_{3,n}(s)  \bigr\rvert \geq \eps_1} \leq \eps_1 \quad  \text{ for all $n \geq n_1$. }
\end{align*}
This proves \cref{eq-Wconv} for $j = 2$. 

Let $\IE'$ and $\IP'$ be the expectation operator and probability operator conditional on $(D'_n(s))_{s \geq 0}$.  
By \cref{lem7,lem3.5}, for each $s \geq 0$, the multigraph corresponding to $\kappa_n(s)$ has the same distribution as the configuration model with the degree sequence $D_{n}'(s)$. 
Noting that $U_n(s)$ is $\sigma(D_n'(s))$-measurable, and observing that  $|t_F^{\ind}(\kappa_n(s))| \leq 1$ and $|U_n(s)| \leq 1$, we obtain
\begin{align}
	\label{eq-3.8-111}
		\MoveEqLeft[0]\sup_{s \geq 0} \lvert \IE'  t_{F}^{\ind}( \kappa_n(s) )  - U_n(s) \rvert \nonumber \\
		& \leq \sup_{s \geq 0} \Bigl\lvert \IE'  \bbclc{t_{F}^{\ind}( \kappa_n(s) )  - U_n(s) \given \max_{1 \leq i \leq n} D'_{n,i}(s) \leq  (n Y_n^{1/2}(s)(\log n)^{2}) x} \Bigr\rvert \nonumber \\
		& \quad + 2 \sup_{s \geq 0} \II \Bigl[ \max_{1 \leq i \leq n} D'_{n,i}(s) \geq  (n Y_n^{1/2}(s)(\log n)^{2}) x \Bigr].
\end{align}
For the first term of the right hand side of \cref{eq-3.8-111}, and recalling that $Y_n(s) \geq a$, we have $n^2Y_n(s) \geq n$ for $n \geq a^{-1}$. Choosing $x = 60 a^{1/2} /(\theta \log n)$, we have if $\max_{1 \leq i \leq n} D'_{n,i}(s) \leq  (n Y_n^{1/2}(s)(\log n)
^{2}) x$, then the conditions in \cref{t3.1} are satisfied. 
Note that $U_n(s)$ can be rewritten as the second term of the left hand side of \cref{eq-t31-1} by replacing $Y_n$ by $Y_n(s)$, $D_n$ by $D_n'(s)$ and $Z_{n,i}$ by $\bar D_{n,i}(s)$.
Then, by \cref{eq-t31-1,lem7}, and noting that the function $p$ is continuous, we have with $x = 60 a^{1/2}/(\theta \log n)$, for $n \geq a^{-1}$,
\begin{multline}
	\label{eq-3.8-112}
	\sup_{s \geq 0} \Bigl\lvert \IE'  \bbclc{t_{F}^{\ind}( \kappa_n(s) )  - U_n(s) \given \max_{1 \leq i \leq n} D'_{n,i}(s) \leq  (n Y_n^{1/2}(s)(\log n)^{2}) x} \Bigr\rvert \\
	\leq C_1 n^{-1/4},
\end{multline}
where $C_1 > 0$ is a constant depending on $a,k, \theta$ and the multigraph $F$. 

For the second term of the right hand side of \cref{eq-3.8-111}, note that for any $0 < u < \theta$, 
\begin{equ}
	\IE^* e^{u \zeta_{n,i}(s)}
	& = \biggl( \frac{\theta}{nY_n(s) (1 - e^{u / (n Y_n(s))}) + \theta} \biggr)^{\theta}\\
	& \leq \biggl( 1 - \frac{u}{\theta} - \frac{u^2}{2 \theta n Y_n(s)}  \biggr)^{-\theta}.
    \label{eq-mg-zeta}
\end{equ}
By \cref{11l-c,11l-2} and the fact that $\inf_{s \geq 0}Y_n(s) \geq a$, taking $\lambda = \theta / ( 6 n Y_n(s) )$ and $x = 60 a^{1/2}/(\theta \log n)$, we have for any $s \geq 0$, 
\begin{equation}
	\begin{split}
		\MoveEqLeft \IP^* \Bigl[ \max_{1 \leq i \leq n} D'_{n,i}(s) \geq  (n Y_n^{1/2}(s)(\log n)^{2}) x \Bigr] \\ 
		  & \leq C_2 n^{5/2} \sum_{i = 1}^n \IP^* \bigl[  Z_{n,i}(s) \geq    (n Y_n^{1/2}(s) (\log n)^{2}) x \bigr] \\
		  & \leq C_2 n^{5/2} \sum_{i = 1}^n e^{ - \lambda (n Y_n^{1/2}(s) (\log n)^{2}) x    }\IE^* {e^{\lambda Z_{n,i}(s)}}  \\
		  & \leq C_2 n^{5/2} \sum_{i = 1}^n e^{ - \lambda (n Y_n^{1/2}(s) (\log n)^{2}) x    }\IE^*   e^{\theta \zeta_{n,i}(s)/6}  \\
		  & \leq C_3 n^{7/2} e^{ -\theta (\log n)^{2} x/(6a^{1/2})} =  C_3 n^{-13/2},  
	\end{split}
	\label{eq-3.8-113}
\end{equation}
where we used \cref{eq-mg-zeta} in the last line, and $C_2$ and $C_3$ are positive constants depending only on $\theta$. 
Therefore, by \cref{eq-3.8-112}, for any $\eps_2 > 0$ and $T > 0$, we have as long as $n > (C_1 / \eps_2)^{4}$, 
\begin{align*}
	\begin{split}
		\MoveEqLeft
		\II \bbbcls{ \sup_{0\leq s \leq T} \bigl\lvert W_{3,n}(s) - W_{4,n}(s) \bigr\rvert \geq \eps_2 } \\
		& \leq \sum_{m = 1}^{\lfloor n^{4 }p_1^{-1} T \rfloor + 1} \II \bbbcls{\max_{1 \leq i \leq n} D'_{n,i}(s) \geq  (n Y_n^{1/2}(s)(\log n)^{2}) x}.
	\end{split}
\end{align*}
Taking expectation on both sides and by \cref{eq-3.8-113} yields 
\begin{align*}
	\IP \bbbcls{ \sup_{0\leq s \leq T} \bigl\lvert W_{3,n}(s) - W_{4,n}(s) \bigr\rvert \geq \eps_2 }
	& \leq C_3 ( p_1^{-1} T + 1 )  n^{-5/2}, 
\end{align*}
for $n \geq (C_1/\eps_2)^4$, which proves \cref{eq-Wconv} for $j = 3$.

For any $\eps_3 > 0$, 
by \cref{lem6}, 
there exists $C_4 > 0$ depending on $k$ such that for $n \geq 4(a + \rho_0)$, 
\begin{equation}
	\begin{split}
		\label{eq:exp-inequality0}
		& \IP \bbbcls{ \sup_{0\leq s \leq T} \bigl\lvert W_{4,n}(s) - W_{5,n}(s) \bigr\rvert \geq \eps_3} \\
		& = \IP \bbbcls{ \sup_{0\leq s \leq T} \bigl\lvert t_F^{\ind} (\kappa_n(s)) - \IE' \{ t_F^{\ind} (\kappa_n(s)) \}\bigr\rvert \geq \eps_3} \\
		& \leq  \sum_{m = 0 }^{ \lfloor n^{4 }p_1^{-1} T \rfloor + 1 } \IP \bbcls{  \bigl\lvert t_F^{\ind} (G_n(m)) - \IE' \{ t_F^{\ind} (G_n(m)) \}\bigr\rvert \geq \eps_3  \given \sup_{0 \leq s \leq T} Y_n(s) \leq n } \\
		& \qquad 	 + \IP \bbbcls{ \sup_{0 \leq s \leq T} Y_n(s) > n } \\
		& \leq 2 \sum_{m = 0 }^{ \lfloor n^{4 } p_1^{-1} T \rfloor + 1 } \exp \bigl( - C_4 n \eps_3^2 / (T + 1) \bigr) + 4 \exp \biggl( - \frac{8n^2}{p_1^{-1}(T + 1)} \biggr)\\
		& \leq 2 \bigl( n^{4 } + 1 \bigr) \bigl( T + 1 \bigr) \exp \bigl(-C_4 n \eps_3^2 / (T + 1) \bigr) +  4 \exp \biggl( - \frac{8n^2}{p_1^{-1}(T + 1)} \biggr). 
	\end{split}
\end{equation}
Then, there exists an $n_3 \geq 0$ depending on $\eps_3, T, p_1$ and $C_4$ such that for all $n \geq n_3$, the right hand side of \cref{eq:exp-inequality0} can be bounded by $\eps_2$. This proves \cref{eq-Wconv} for $j = 4$ and hence the tightness of $(t_{F}^{\ind}(\kappa_n))_{n \geq 1}$.

{\medskip \noindent \it Step 3. Finite dimensional convergence}.
Recalling that $Y_n \to Y$ $\IP$-a.s.\ as $n \to \infty$, and by \cref{eq-Wconv}, we have for any $F \in \cM$ and $s \geq 0$, 
\begin{align*}
	t_F^{\ind}(\kappa_n(s)) \longto t_F^{\ind}(\kappa(s)) \quad \text{ in probability\ as $n \toinf$. }
\end{align*}
Then, for any $F_1, \dots, F_q \in \cM$ and $0 \leq s_1 < \dots < s_q < \infty$, we have 
\begin{align*}
	\prod_{j = 1}^q t_{F_j}^{\ind}(\kappa_n(s_j)) \longto \prod_{j = 1}^q t_{F_j}^{\ind}(\kappa(s)) \quad \text{in probability \ as $n \toinf$.}
\end{align*}
By the boundedness property of $t_F^{\ind}$ and bounded convergence theorem, we have 
\begin{align*}
	\lim_{n \toinf} \IE \biggl\{ \prod_{j = 1}^q t_{F_j}^{\ind}(\kappa_n(s_j)) \biggr\}
	= \IE \biggl\{ \prod_{j = 1}^q t_{F_j}^{\ind}( \kappa(s) ) \biggr\}	. 
\end{align*}
This completes the proof. 
\end{proof}

\section*{Acknowledgements}

This project was supported by the Singapore Ministry of Education Academic Research Fund Tier 2 grant MOE2018-T2-2-076.

\setlength{\bibsep}{0.5ex}
\def\bibfont{\small}


\end{document}